\documentclass[a4paper, 11pt]{article}

\usepackage{amsmath,amsfonts,amssymb,amsthm}
\usepackage[a4paper, margin=1in]{geometry}

\usepackage[T1]{fontenc}

\usepackage{lmodern}
\usepackage{calc}
\usepackage[dvipsnames]{xcolor}
\usepackage{nameref}
\usepackage{upgreek}     
\usepackage{enumitem}
\usepackage{graphicx}
\usepackage{mathrsfs}
\usepackage{tikz-cd}
\usepackage{hhline}
\usepackage{textcomp}
\usepackage{amscd}
\usepackage[all,cmtip]{xy}
\usepackage{mathtools}
\usepackage{extarrows}
\usepackage{upgreek}
\usepackage{bbm}
\usepackage[displaymath]{lineno}
\usepackage[singlespacing]{setspace}%
\usepackage[colorlinks=true,breaklinks=true,bookmarks=true,urlcolor=blue,
citecolor=BrickRed,linkcolor=teal,bookmarksopen=false,draft=false]{hyperref}
\providecommand{\U}[1]{\protect\rule{.1in}{.1in}}

\makeatletter
\let\orgdescriptionlabel\descriptionlabel
\renewcommand*{\descriptionlabel}[1]{%
	\let\orglabel\label
	\let\label\@gobble
	\phantomsection
	\edef\@currentlabel{#1}%
	\let\label\orglabel
	\orgdescriptionlabel{#1}%
}
\makeatother

\theoremstyle{plain}
\newtheorem{theorem}{Theorem}[section]
\newtheorem{lem}[theorem]{Lemma}
\newtheorem{cor}[theorem]{Corollary}
\newtheorem{Exmp}[theorem]{Example}

\newtheorem{Rem}[theorem]{Remark}
\theoremstyle{definition}
\newtheorem{Defn}[theorem]{Definition}
\newtheorem*{definition*}{Definition} 
\newcommand{\RR}{\mathbb{R}}
\newcommand{\NN}{\mathbb{N}}

\sbox0{$\fontdimen8\textfont3=0.4pt$}
\newcommand{\ol}[1]{\overline{#1}}

\newcommand{\Id}{\mathop\mathrm{Id}\nolimits}

\renewcommand{\emptyset}{\varnothing}

\DeclareMathAlphabet{\mathpzc}{OT1}{pzc}{m}{it}

\newcommand{\va}{\varphi}

\newcommand{\vf}{\mathcal{V}}
\newcommand{\dd}{{\tt D}}
\newcommand{\dt}[1]{{\tt d}{#1}}
\newcommand{\fs}[1]{\mathsf {#1}}
\newcommand{\eu}[1]{\EuScript {#1}}

\newcommand{\mt}{\mathbbm {d}}

\newcommand{\set}[1]{\left\{#1\right\}}
\newcommand\Set[2]{\left\{#1\mid#2\right\}} 
\newcommand{\snorm}[2][]{\left\lVert#2\right\rVert_{#1}}

\newcommand{\Sem}[1]  {\textsf{Sem}(#1)}

\newcommand{\zero}[1]{\boldsymbol{0}_{#1}}

\usepackage{bm}    

\newcommand{\li}{Lipschitz}
\newcommand{\rr}{\mathbb{R}}
\newcommand{\nn}{\mathbb{N}}

\newcommand{\mc}[1]{\mathsf{MC}^{#1}}
\newcommand{\bl}[1] {\boldsymbol {#1}}

\DeclareMathOperator{\codim}{codim}

\DeclareMathOperator{\Aut}{Aut}

\DeclareMathAlphabet\EuScript{U}{eus}{m}{n}
\SetMathAlphabet\EuScript{bold}{U}{eus}{b}{n}
\newcommand\opn{\ensuremath{\mathrel{\mathpalette\opncls\circ}}}

\newcommand{\opncls}[2]{
	\ooalign{$#1\subseteq$\cr
		\hidewidth\raisefix{#1}\hbox{$#1{\stylefix{#1}#2}\mkern2mu$}\cr}}
\def\raisefix#1{
	\ifx#1\displaystyle
	\raise.39ex
	\else
	\ifx#1\textstyle
	\raise.39ex
	\else
	\ifx#1\scriptstyle
	\raise.275ex
	\else
	\raise.150ex
	\fi
	\fi
	\fi
}
\def\stylefix#1{
	\ifx#1\displaystyle
	\scriptstyle
	\else
	\ifx#1\textstyle
	\scriptstyle
	\else
	\ifx#1\scriptstyle
	\scriptscriptstyle
	\else
	\scriptscriptstyle
	\fi
	\fi
	\fi
}
\DeclareFontFamily{U}{mathx}{\hyphenchar\font45}
\DeclareFontShape{U}{mathx}{m}{n}{
	<5> <6> <7> <8> <9> <10>
	<10.95> <12> <14.4> <17.28> <20.74> <24.88>
	mathx10
}{}
\newcommand{\fr}{Fr\'{e}chet }

\newcommand{\TM}[1]{\mathrm{T}\mathsf{#1}}       % Tangent bundle
\newcommand{\TS}[1]{\mathrm{T}_{#1}}   
\newcommand{\TTM}[1]{\mathrm{T}(\mathrm{T}\mathsf{#1})} 
\newcommand{\TpM}[2]{\mathrm{T}_{#1}\mathsf{#2}}  
\DeclareMathAlphabet{\mathsfit}{OT1}{cmss}{m}{sl}
\newcommand{\s}{\eu{S}}
\DeclareMathOperator{\supp}{supp}
\newcommand{\Adm}[2]{{A}_{#1, #2}}
\newcommand{\GeoFlow}[1]{\Phi_{#1}}
\newcommand{\subjclass}[1]{\textbf{AMS Subject Classifications (2020):} #1\par}
\newcommand{\keywords}[1]{\textbf{Keywords:} #1\par}
\title{Spray-Invariant Sets in Infinite-Dimensional Manifolds}
\author{Kaveh Eftekharinasab}

\date{}
\usepackage{cite}
\begin{document}
	
	\maketitle
	
	\begin{abstract}
		We introduce the concept of spray-invariant sets on infinite-dimensional manifolds, where any geodesic of a spray starting in the set stays within it for its entire domain. These sets, possibly including singular spaces such as stratified spaces, exhibit different geometric properties depending on their regularity:
		sets that are not differentiable submanifolds  may show sensitive dependence, for example, on parametrization, whereas for differentiable submanifolds invariance is preserved under reparametrization.
		
		This framework offers a broader perspective on geodesic preservation than the rigid notion of totally geodesic submanifolds, with examples arising naturally even in simple settings, such as linear spaces equipped with flat sprays.
	\end{abstract}
	\let\thefootnote\relax\footnotetext{
		\,\,\,\,\,\keywords{ Sprays, geodesics, spray-invariant sets, infinite-dimensional manifolds, adjacent and second-order adjacent cone, singular spaces.}
		This work was supported by grants from the Simons Foundation (1030291, 1290607, K.A.E)}

	\section*{Introduction}\label{sec1}

This work studies subsets of infinite-dimensional manifolds, including singular spaces such as stratified spaces, where any geodesic of a spray starting in the subset remains within it for the entire duration of its definition. The behavior of such sets, which we call \emph{spray-invariant}, depends strongly on their regularity. For instance, for sets that are not differentiable submanifolds, reparametrization of geodesics may affect whether they remain within the set. In contrast, for differentiable submanifolds, this invariance is preserved. The motivation for studying spray-invariant sets with less regularity stems from the observation that such sets can arise naturally even in simple settings like linear spaces equipped with flat sprays.

We focus on the intrinsic properties of sprays and work within the broader context of spray geometry. This approach does not require the existence of a spray induced by a Finsler  (or Riemannian) metric or compatibility with such a structure.  Consequently, we can analyze the dynamics
of geodesics in the setting of infinite-dimensional manifolds, where traditional Finsler (or Riemannian) geometric tools are either unavailable or inapplicable.
We primarily focus on the more general context of \fr manifolds; however, our results are  applicable to Hilbert and Banach manifolds as well. 

Given a subset \( S \) of a manifold \( \fs{M} \) and a spray \( \s \) on \( \fs{M} \), we define the \textit{admissible set} \( A_{\s, S} \) (Definition \ref{def:admiset}) as the collection of all  tangent vectors \( v \in  \mathrm{T}\fs{M} \) such that  the projection \( \tau(v) \in S \), and \( \s(v) \) belongs to the second-order adjacent cone of \( S \) at \( \tau(v) \). In Theorem \ref{th:1}, we prove that if \(S\) is closed, then a geodesic \( g(t) \) lies entirely in \( S \) if and only if its tangent vector \( g'(t) \) belongs to \( A_{\s, S} \) for all \( t \) in its domain. This equivalence establishes \( A_{\s, S} \) as a fundamental invariant for analyzing the behavior of geodesics. Building on this, we define a spray-invariant set as follows: a subset \( S  \) is spray-invariant for the spray \( \s \) if, for every geodesic \( g\colon I \to \fs{M} \) of \( \s \) with initial tangent \( g'(0) \in A_{\s, S} \), the entire trajectory remains within \( S \), i.e., \( g(t) \in S \) for all \( t \in I \), where \( I \) is the maximal interval of existence.  Example \ref{ex:1}  provides an  instance where the spray-invariant sets is a singular space. In Example \ref{ex:stra}, we present an instance of stratified spray-invariant set.

For a sufficiently differentiable submanifold \( S \), the admissible set \( A_{\s, S}  \) provides a  characterization of totally geodesic submanifolds. Specifically, in Theorem  \ref{th:sub}, we prove that \( A_{\s, S}  =\mathrm{T}S\) if and only if \(S\) is a totally geodesic submanifold. We  apply this theorem in Example~\ref{ex:hi} to the infinite-dimensional manifold of loops on a sphere.
Theorem  \ref{th:sub} yields a geometric criterion for identifying totally geodesic structures: that is, if \( S \) is closed and {locally geodesically convex} (i.e., every pair of sufficiently close points in \( S \) is connected by a unique geodesic segment lying entirely in \( S \)), then \( S \) is totally geodesic (Corollary \ref{cor:1}). Using this criterion,  Example \ref{ex:crit} presents a totally geodesic submanifold. In contrast, Example~\ref{ex:2} provides a differentiable submanifold that is spray-invariant but not totally geodesic.

In Subsection \ref{subsec:1}, we introduce the notion of spray automorphisms and establish, in Theorem \ref{th:aut}, that the image of a spray-invariant set under such an automorphism remains spray-invariant.  Example \ref{ex:9} illustrates this with the flat spray on \( C^\infty(\mathbb{R}, \mathbb{R}) \) and a singular spray-invariant set.
In Subsection \ref{sub:orbit},  we study Lie group actions on smooth  manifolds and their orbit type decompositions. We show that if the action admits suitable local slices, then each orbit type stratum is invariant under a group-invariant spray
(Theorem \ref{th:decom}).  

If \( S \) is a spray-invariant set, a natural question arises: does the spray \( \s \), when regarded as a first-order vector field on \( \TM{M} \), remain second-order adjacent tangent to \( A_{\s, S}  \)? This reformulation reduces the problem from analyzing second-order dynamics on \( \fs{M} \) to studying first-order dynamics on \( \TM{M} \), which may be more tractable.  
This question can be addressed using the Nagumo-Brezis Theorem, which provides a criterion for determining the invariance of sets under vector fields.
However, the theorem’s classical formulation applies primarily to Banach manifolds and does not generalize straightforwardly to arbitrary \fr manifolds. For a detailed discussion of these limitations and potential adaptations, see \cite{kr}.

In Section \ref{sec:mc}, we revisit the category of {\(\mc{k}\)-Fr\'{e}chet manifolds}, where the Nagumo-Brezis Theorem holds under nuclearity assumptions. For a nuclear \( \mc{k}\)-Fr\'{e}chet manifold \( \fs{M} \) and a closed subset \( S \subset \fs{M} \), we prove (Theorem \ref{th:imp}) that
\( S \) is  spray-invariant  for the spray \( \s \) if and only if 
\( \s \), regarded  as a first-order vector field on \( \TM{M} \), is  second-order adjacent tangent to \( A_{\s, S}  \).   

A key property of this class  of manifolds is the validity of the transversality theorem.
Using this, we give a transversality-based criterion to characterize spray-invariant sets (Theorem \ref{th:tf}).

In Section \ref{sec:hb}, we consider Banach and Hilbert manifolds. All results from Sections \ref{sec:2} and \ref{sec:mc}  remain valid with appropriate modifications to their assumptions.  
\section{Sprays}
We employ the notion of differentiable mappings, known as \( C^k \)-mappings in the Michal–Bastiani sense or Keller's \( C_c^k \)-mappings. 

Throughout this paper, we assume that \( (\fs{F}, \Sem{\fs{F}}) \) and \( (\fs{E}, \Sem{\fs{E}}) \) are \fr spaces over \( \rr \), where \( \Sem{\fs{F}} = \Set{\snorm[\fs{F},n]{\cdot}}{n \in \nn} \) and \( \Sem{\fs{E}} = \Set{\snorm[\fs{E},n]{\cdot}}{n \in \nn} \) are families of continuous seminorms that define the topologies of \( \fs{F} \) and \( \fs{E} \), respectively. 
We use the notation \( U \opn \mathsf{T} \) to denote that \( U \) is an open subset of the topological space \( \mathsf{T} \).
\begin{Defn}[Definition I.2.1, \cite{neeb}]\label{def:diff}
	Let  $ \va\colon U \opn \fs{E}  \to  \fs{F}$ be a mapping. Then the  derivative
	of $\va$ at $x$ in the direction $h$ is defined by 
	\[
	\dd \va_x(h)=\dd\va(x)(h) \coloneqq
	\lim_{t \to 0} {1\over t}(\va(x+th) -\va(x))
	\]
	whenever it exists. 
	The function $\va$ is called differentiable at
	$x$ if $\dd \va(x)(h)$ exists for all $h \in \fs{E}$. 
	It is called 
	continuously differentiable if it is differentiable at all
	points of $U$, and the mapping
	\[
	\dd \va \colon U \times \fs{E} \to \fs{F}, \quad (x,h) \mapsto \dd \va(x)(h)
	\]
	is continuous. 
	It is called a $C^k$-mapping, $k \in \nn \cup \{\infty\}$, 
	if it is continuous, the iterated directional derivatives 
	\(
	\dd^{j}\va_{x}(h_1,\ldots, h_j)=\dd^{j}\va(x)(h_1,\ldots, h_j)
	\)
	exist for all integers $j \leq k$, $x \in U$, and $h_1,\ldots, h_j \in \fs{E}$, 
	and all mappings $\dd^j \va \colon U \times \fs{E}^j \to F$ are continuous. Alternatively, we  refer to $C^{\infty}$-mappings as being smooth.
\end{Defn}
In light of the  chain rule for $C^k$-mappings between open subsets of \fr spaces 
(see \cite[Proposition I.2.3]{neeb}), we can naturally define $C^k$-manifolds modeled on \fr spaces.  We assume that these \fr manifolds are Hausdorff.

Henceforth, we assume that \( \fs{M} \) is a \( C^k \)-\fr manifold modeled on \( \fs{F} \), \(k \geq 4\).
Recall that the tangent space 
\( \TpM{p}{M}\) at a point \( p \in \fs{M} \) is defined as the space of equivalence classes of tangent curves at \( p \) (see \cite[I.3.3]{neeb}).
The tangent bundle  \( \tau\colon \mathrm{T}\fs{M} \to \fs{M} \) is a \( C^{k-1} \)-\fr manifold modeled on \( \fs{F} \times \fs{F} \). 
Given a chart \( (U, \varphi) \) on \( \fs{M} \) with \( \varphi\colon U \to \fs{F} \), the induced chart on \( \TM{M} \) is \( \big(\mathrm{T}U,\mathrm{T}\va \big) \), where \(\mathrm{T}U = \tau^{-1}(U) \) and
\[
\mathrm{T}\va\colon\mathrm{T}U \to \varphi(U) \times \fs{F}, \quad\mathrm{T}\va(p, v) = (\varphi(p), \dd \varphi_p(v)),
\]
for \( p \in U \text{ and } v \in \TpM{p}{M}\).
We will  require the tangent bundle over \(\TM{M}\), commonly called the double tangent bundle, denoted by \(\tau_{2} \colon\TTM{M} \to \TM{M}\). This can result in expressions of considerable complexity. In such cases, we sometimes  use the notation \(\va_* \) to denote the tangent map $\mathrm{T}\va$.
Consider a chart \((U, \varphi)\) on \( \fs{M} \). Then, the tangent map of \({\varphi}_*\) is given by
\begin{gather*}
	\mathrm{T}({\varphi}_*)\colon\mathrm{T}(TU) \to (\varphi(U) \times \fs{F}) \times (\fs{F} \times \fs{F}),\\
	\mathrm{T}({\varphi}_*) \big((p, v), (u, {w})\big) = \left((\varphi(p), \dd\varphi_p(v)), (\dd\varphi_p(u), (\dd^2\varphi_p(v,u)+\dd \va_p({w})))\right), 
\end{gather*}
for\( p \in U \), \(v, u \in\TpM{p}{M}\) and \( {w} \in\mathrm{T}_{{v}}(\mathrm{T}_p\fs{M}) \).

We identify \( U \times \fs{F} \) with \(\mathrm{T}U \) and correspondingly \(\mathrm{T}\va \) with \( \dd \va \). 
Thus, for brevity, we may write \(\mathrm{T}\va \) or \( \va_* \), implicitly understanding this identification.

Consider two overlapping charts \((U, \varphi)\) and \((V, \psi)\) on \(\fs{M}\) with \(U \cap V \neq \emptyset\). 
For \(\TM{M}\), the transition map \(\phi =\psi \circ \varphi^{-1}\) induces the following 
transformation equation:
\begin{equation}\label{eq:1}
	\phi_*(p, v) = \left(\phi(p), \dd \phi_p(v)\right), \quad
	\forall (p, v) \in \va (U \cap V) \times \fs{F}.
\end{equation}
By differentiating \eqref{eq:1}, we derive the following change of coordinates rule for \(\TTM{M}\):
\begin{equation}\label{eq:2}
	\mathrm{T}(\phi_*)\left((p, v), (x,y)\right) = 
	\big(\dd \phi_p(x),
	\dd^2 \phi_p(x,v)+\dd \phi_p(y)\big),
\end{equation}
for all \( (p, v) \in \va (U \cap V) \times \fs{F}\) and all $x,y \in \fs{F} $.

To simplify notations, let $ ( U,\va) $ be a chart on $ \fs{M} $, $p\in U$, $ v \in\TpM{p}{M} $, and $ w \in \TS{v}(\TM{M}) $. We define
\begin{equation}\label{eq:n1}
	v_{\va} \coloneqq \dd \varphi_p(v),  \text{and } w_{\va_*} \coloneqq \dd {(\va_*)}_v (w) =(w_{\va_*,1}, w_{\va_*,2}).
\end{equation}
Here, $w_{\va_*,1}$
and $w_{\va_*,2}$ are the components of $w_{\va_*}$, obtained by applying Equation 
\eqref{eq:1} to the tangent vectors.
Consequently, from Equation 
\eqref{eq:2}  for $ p \in V $, we obtain
\begin{equation}\label{eq:n2}
	w_{\psi_*,2} = \dd^2 \phi_{\va(x)}(v_p, w_{\va_*,1} )+\dd \phi_{\va(x)}( w_{\va_*,2}). 
\end{equation}
The theory of sprays, studied in the context of Banach manifolds by Lang~\cite{lang}, was later generalized to \fr manifolds in \cite{k1, k2} with the aim of investigating the properties of geodesics on these manifolds.

We now recall the definition of sprays and related concepts that will be required.

A $ C^{r}$-mapping $V \colon \TM{M} \to \TTM{M}$, \(1 \leq r \leq k-2\),  satisfying
\(
\tau_{*} \circ V =\Id_{\TM{M}}
\)
is called
a second-order $ C^{r} $-vector field. If, in addition, \(
\tau_2 \circ V =\Id_{\TM{M}}
\), then \(V\) is called symmetric. 
A second-order vector field  is symmetric if and only if its
integral curves are canonical lifts of curves in \(\fs{M}\). 

We will later use the following lemma, which was proved using different arguments for finite-dimensional manifolds in  \cite[Corollary 5.1.6]{spf}.
\begin{lem}\label{lem:4}
	Let \( V\colon \TM{M} \to \mathrm{T}(\TM{M}) \) be a \(C^r\)-symmetric second-order vector field, and let \( \phi \) be a \(C^{r+2}\)-automorphism of $\fs{M}$. Then, \( \phi_{**} \circ V \circ \phi_{*}^{-1} \) is also a \(C^r\)-symmetric second-order vector field.
\end{lem}

\begin{proof}
	Let \( (x, y) \in \TM{M} \) and \( (x, y, X, Y) \in \mathrm{T}(\TM{M}) \). Then, 
	\[
	(x, y) \overset{\phi_*}{\longmapsto} \big(\phi(x), \dd\phi(x)(y)\big), \text{ and}
	\]
	\[
	(x, y, X, Y) \overset{\phi_{**}}{\longmapsto} \big(\phi(x), \dd\phi(x)(y), \dd\phi(x)(X), \dd^2\phi(x)(y, X) + \dd\phi(x)(Y)\big).
	\]
	By applying \( V \) to \( 	\phi_{*}^{-1}(x, y) = (\phi^{-1}(x), \dd\phi^{-1}(x)  (y)) \), we obtain
	\[
	V(\phi_{*}^{-1}(x, y)) = \big(\phi^{-1}(x), \dd\phi^{-1}(x)  (y), \dd\phi^{-1}(x)  (y), Y(\phi^{-1}(x), \dd\phi^{-1}(x)  (y))\big).
	\]
	Here, \( Y(\phi^{-1}(x), \dd\phi^{-1}(x)(y)) \) is a tangent vector to \( \fs{M} \) at \( \phi^{-1}(x) \).
	
	Next, applying \( \phi_{**} \) to \( V(\phi_{*}^{-1}(x, y)) \) yields
	\begin{align*}
		\phi_{**}(V(\phi_{*}^{-1}(x, y))) &= \phi_{**}\big(\phi^{-1}(x), \dd\phi^{-1}(x)(y), \dd\phi^{-1}(x)(y), Y(\phi^{-1}(x), \dd\phi^{-1}(x)(y))\big) \\
		&= \Big(\phi(\phi^{-1}(x)), \dd\phi(\phi^{-1}(x))(\dd\phi^{-1}(x)(y)), \dd\phi(\phi^{-1}(x))(\dd\phi^{-1}(x)(y)), \\
		& \dd^2\phi(\phi^{-1}(x))\big(\dd\phi^{-1}(x)(y), \dd\phi^{-1}(x)(y)\big) + \dd\phi(\phi^{-1}(x))\big(Y(\phi^{-1}(x), \dd\phi^{-1}(x)(y))\big)\Big) \\
		&= \big(x, y, y, Z(x, y) \big),
	\end{align*}
	where  
	\[
	Z(x, y) = \dd^2\phi(\phi^{-1}(x))\big(\dd\phi^{-1}(x)(y), \dd\phi^{-1}(x)(y)\big) 
	+ \dd\phi(\phi^{-1}(x))\big(Y(\phi^{-1}(x), \dd\phi^{-1}(x)(y))\big).
	\]
	is a \(C^r\)-function. 
	The projections \( \tau_* \)  and \(\tau_2\) act as follows
	\[
	\begin{aligned}
		\tau_*(\phi_{**} \circ V \circ \phi_{*}^{-1}(x, y)) &= \tau_*(x, y, y, Z(x, y)) = (x, y), \\
		\tau_2(\phi_{**} \circ V \circ \phi_{*}^{-1}(x, y)) &= \tau_2(x, y, y, Z(x, y)) = (x, y).
	\end{aligned}
	\]
	Thus, \( \phi_{**} \circ V \circ \phi_{*}^{-1} \) is a $C^r$-symmetric second-order vector field.
\end{proof}
Assume that  $ s $ is a fixed real number, and define the mapping
\begin{equation*}
	L_{\TM{M}}\colon \TM{M} \to \TM{M},\quad v \mapsto sv.
\end{equation*}
Then,  the induced map $ (L_{\TM{M}})_*\colon \TTM{M} \to \TTM{M} $ satisfies
\begin{equation*}
	(L_{\TM{M}})_* \circ L_{\TTM{M}}= L_{\TTM{M}} \circ (L_{\TM{M}})_*,
\end{equation*}
which follows from the linearity of $ L_{\TM{M}} $ on each fiber.
A second-order symmetric $ C^{r} $-vector filed $\s \colon \TM{M} \to \TTM{M}$  is called a {spray} if 
it satisfies the following condition:
\begin{enumerate}[label=$ {\bf (SP\arabic*)} $,ref=SP\arabic*]
	\itemindent=10pt
	\item \label{eq:sp1} $\s(sv) = (L_{\TM{M}})_*(s\s(v))$ for
	all $ s \in \rr $ and $ v \in \TM{M} $.
\end{enumerate}

A manifold that possess a $ C^k $-partition of unity admits a spray of class \(C^{k-2}\).  Important examples are Lindel\"{o}f manifolds modelled on nuclear \fr spaces, cf.
\cite[Theorem 16.10]{km}. 
Since we require that  sprays be of class at least \(C^2\), the underlying manifolds must be of class at least \(C^4\). Therefore, we assume henceforth that \(\fs{M}\) is at least of class  \(C^4\).

Let \( \gamma\colon I \subseteq \rr \rightarrow \fs{M} \) be a \( C^r \)-curve, \( r \geq 2 \). A lift of \( \gamma \) to \( \TM{M} \) is a curve \( \widetilde{\gamma}\colon I \rightarrow \TM{M} \) such that \( \tau \circ \widetilde{\gamma} = \gamma \). In other words, a lift of a curve is a curve in the tangent bundle that projects down to the original curve on the base manifold. The curve \( \gamma' = \dd \gamma\colon I \rightarrow \TM{M} \) is called the canonical lift of \( \gamma \).
An integral curve \( \eta \) of a spray \( \s \) is a curve in \( \TM{M} \) such that \( \eta'(t) = \s(\eta(t)) \).
Each integral curve \( \eta \) of \( \s \) is the canonical lift of \( \tau(\eta) \), i.e., \( (\tau(\eta))' = \eta \). For any \( t \) in the domain of \( \eta \), the latter formula reads as \( (\tau(\eta))'(t) = \eta(t) \).

A curve \( g\colon I \subseteq \rr  \rightarrow \fs{M} \) is called a geodesic of a spray \( \s \) if its canonical lifting \( g'\colon I \rightarrow \TM{M} \) is an integral curve of the spray \( \s \). Since \( g' \) lies above \( g \) in \( T \fs{M} \), that is, \( \tau(g') = g \), we can express the geodesic condition by
\begin{equation}\label{eq:gd}
	g'' = \s(g').
\end{equation}

To avoid ambiguity, when necessary, we will denote the local representations of objects in a chart $ (U,\va ) $ of $ \fs{M} $ by a subscript $ \va $.
The  local representations of $ L_{\mathrm{T}U} $ and $ (L_{\mathrm{T}U})_*$ in   $ ( U,\va) $  are given by 
\begin{equation*}
	L_{\mathrm{T}U}:(x,v) \mapsto (x,sv) \quad \text{and}\quad (L_{\mathrm{T}U})_*\colon(x,v,u,w)\mapsto (x,sv,u,sw).
\end{equation*}
Therefore, we get
\(
	L_{\mathrm{T}(\mathrm{T}U)} \circ (L_{\mathrm{T}U})_*(x,v,u,w)=(x,sv,su,s^2w).
\)
Let $ {{\s}}_{\va}= ({{\s}}_{\va,1},{{\s}}_{\va,2})\colon (U \times \fs{F} ) \to \fs{F} \times \fs{F}$
be a local representation of $ \s $, where each  $ {{\s}}_{\va,i} $ maps $ U \times \fs{F} $
to $ \fs{F} $ with $ {{\s}}_{\va,1}(x,v)=v $. Then,  for
all $ s \in \rr $, the following condition holds:
\begin{equation}\label{eq:123}
	{{\s}}_{\va,2}(x,sv)=s^2{{\s}}_{\va,2}(x,v).
\end{equation}
Thus,  condition  \eqref{eq:sp1} not only characterizes a second-order vector field but also implies that $\s_{\va,2}$ is homogeneous of degree 2 in  $v$. Consequently, $\s_{\va,2}$ is a quadratic map in its second variable, i.e.,
\begin{equation*}
	{{\s}}_{\va,2} (x,v) = \dfrac{1}{2}\dd_2^2 {{\s}}_{\va,2} (x, \zero{\fs{F}})(v,v)
\end{equation*}
where $ \dd_2^2 $ is the second partial derivative with respect to the second variable.
In the chart, a geodesic \(g\) of \(\s\) has two components:
\(
{g}(t)=\big( x(t),v(t)\big) \in U \times\fs{F}.
\)
Accordingly,  Equation \eqref{eq:gd} takes the  form
\begin{equation}\label{eq:gd2}
	\dfrac{\dt{x}}{\dt{t}} = v(t),\quad \dfrac{ \mathtt{d}^2{x}}{\dt{t^2}} = 	{{\s}}_{\va,2} (x,v(t)) = \dfrac{1}{2}\dd_2^2 {{\s}}_{\va,2} (x, \zero{\fs{F}})(v(t),v(t)).
\end{equation}
\begin{Defn}
	\label{def:proj_equiv}
	Two sprays $S$ and $\bar{S}$ on a manifold $\fs{M}$ are said to be {projectively equivalent} if they share the same geodesics as point sets. 
Specifically,  for any geodesic \(\ol{g}\) of \(\ol{\s}\),
there exists an orientation-preserving reparametrization \(\ol{t}=\ol{t}(t)\) such that the curve \(g(t) \coloneqq \ol{g}(\ol{t}(t))\) is a geodesic of \(\s\), and vice versa.
\end{Defn}
Suppose  \(\s\)  is projectively equivalent to \(\ol{\s}\). For any \(v \in \mathrm{T}_x\fs{M}\), let \(g(t) \) be a  geodesic  of \(\s\) with \(g(0)=x \) and \(g'(0)=v \). Then, there exists a reparametrization  \(\ol{t}=\ol{t}(t)\) with \(\ol{t}(0)=0\) and \((\ol{t})'(0)=1\), such that  \(  \ol{g}(\ol{t}) \coloneqq g(t)\) is the geodesic of \(\ol{\s}\) satisfying \(\ol{g}(0)=x \) and \((\ol{g}){'}(0)=v \).

By definition, the second derivative of the coordinate representation of the geodesic at \(t=0\) is 
\( g_{\va}''(0)= \frac{ \mathtt{d}^2{x}}{\dt{t^2}}|_{t=0}\). Therefore, 
Equation \eqref{eq:gd2} implies
\begin{equation}\label{eq:gd3}
	{{\s}}_{\va,2} (x,v_{\va}) = g_{\va}''(0)= (\ol{g}_{\va})\mkern1mu''(0)+ (\ol{t})''(0)(\ol{g}_{\va})'(0)
	= {\ol{\s}}_{\va,2} (x,v_{\va}(t)) + (\ol{t})''(0)v_{\va}.
\end{equation}
Here, the final term \((\ol{t})''(0)v_{\va}\)  is a scalar multiplication, where the real number
\((\ol{t})''(0)\)  multiplies the vector \(v_{\va}\), which is the local representation of the tangent vector in \(\fs{F}\).
Letting 
\(
P(x,v_{\va}) \coloneqq (\ol{t})''(0)
\), we observe that \(P\)
depends only on \(x,v_{\va}\).  Furthermore, $P$ satisfies the homogeneity
\[
P(x, rv_{\va}) = r P(x, v_{\va}), \quad \forall r \in \rr.
\]
which follows from the quadratic homogeneity of sprays. Thus, 
\begin{equation}\label{eq:gd4}
	{{\s}}_{\va,2} (x,v_{\va}) 
	= {\ol{\s}}_{\va,2} (x,v_{\va}) + P(x,v_{\va})v_{\va}.
\end{equation}

Conversely, suppose that \(\s\) and \(\ol{\s}\) satisfy Equation \eqref{eq:gd4} with $P$ homogeneous of degree 1 in $v$. Given a geodesic \(g(t)\) of \(\s\), the reparametrization  \(\ol{t}(t)\) can be constructed by solving  \(\ol{t}''(t)=P(g(t),g'(t)))\) with \(\ol{t}(0)=0\) and \((\ol{t})'(0)=1\), implying
\(  \ol{g}(\ol{t}) = g(t)\) is a geodesic of \(\ol{\s}\).

Sprays that are projectively equivalent form equivalence classes, which we call projective sprays. For a spray $\s$, its corresponding equivalence class is denoted by $[\s]$.

\begin{Rem}
	Vector fields on general \fr manifolds may lack integral curves, and even when they exist, uniqueness is not guaranteed. Consequently, a geodesic flow may fail to exist or be well-defined. However, our study remains unaffected by these limitations, as our primary focus is the dynamics of geodesics, independent of their existence or uniqueness.
\end{Rem}

\section{Spray-Invariant sets}	\label{sec:2}
Sets invariant under the flow of vector fields have been extensively studied and well-documented for Banach manifolds in \cite{mon1}. Partial generalizations to \fr manifolds were subsequently established in \cite{kr}. In this section, drawing inspiration from the concept of flow-invariant sets, we introduce the notion of spray-invariant sets with respect to a spray on \fr manifolds.

As our aim is to define spray-invariant sets that are not necessarily submanifolds, we require the notions of tangent and second-order tangent cones. However, the concept of a tangent cone to a subset of a topological vector space can be formulated in various ways.
We adopt the adjacent cone (also known as the intermediate cone) as defined in \cite[Definition 4.1.5]{aub}. 

In \fr spaces, convergence occurs if and only if it occurs with respect to each  seminorm. Therefore,  a sequence converges to a set if and only if all pseudo-distances between the sequence and the set simultaneously approach zero. The pseudo-distance of an element $x \in \fs{F}$ to a subset $S \subset \fs{F}$ with respect to a seminorm $\snorm[\fs{F},n]{\cdot}$ is defined by
\begin{equation*}
	\mathrm{d}_{\fs{F},n} (x,S) \coloneqq \inf \Set{\snorm[\fs{F},n]{x-y}}{y \in S}. 
\end{equation*}	
\begin{Defn}\label{def:2}
	Let \( \emptyset \neq S \subset \fs{F}\) and $ s \in {S} $. The  adjacent cone
	$ \mathrm{T}_sS$ is defined by
	\begin{equation*}
		\mathrm{T}_sS\coloneqq \Set{f \in \fs{F}}
		{\lim_{t \to 0^+} t^{-1}\mathrm{d}_{\fs{F},n} \left(s + t f, S\right)=0, \forall n \in \NN}.
	\end{equation*}
\end{Defn}
The adjacent cone $\mathrm{T}_sS$ is nonempty and closed. The proof is a straightforward adaptation of the arguments given in the Banach space case (cf.~\cite[Proposition~1.2]{mon1}).

Intuitively, the adjacent cone $\mathrm{T}_sS$ at a point $s \in S$ consists of all vectors \(f\) representing permissible directions of motion from \(s\), i.e., directions in which one can depart from $s$ while remaining arbitrarily close to  $S$. This idea is formalized by the condition that all pseudo-distances from  $s + t f$  to $S$ must vanish faster than the step size $t$. As we will see, if \(S\) is a differentiable submanifold, then $\mathrm{T}_sS$ is the tangent space at \(s\).

\begin{Exmp}
	\label{ex:adjacent_cone}
	Let $\fs{F}$ be the \fr space $\mathbb{R}^\infty$ of all real sequences, with the topology given by the family of seminorms $\snorm[\fs{F},n]{x_n} = |x_n|$ for $n \in \mathbb{N}$. Consider the set
	\[
	S = \{(x_i) \in \fs{F} \mid x_i \geq 0 \text{ for all } i \in \mathbb{N}\}
	\]
	and let $s$ be the zero sequence. By definition, a vector $f = (f_i) \in \fs{F}$ belongs to $\mathrm{T}_sS$ if and only if for every $n \in \mathbb{N}$,
	\[
	\lim_{t \to 0^+} t^{-1} \mathrm{d}_{\fs{F},n}(s + tf, S) = 0.
	\]
	Since $s = 0$, this simplifies to $\lim_{t \to 0^+} t^{-1} \mathrm{d}_{\fs{F},n}(tf, S) = 0$. The pseudo-distance with respect to the $n$-th seminorm is given by
	\[
	\mathrm{d}_{\fs{F},n}(tf, S) = \inf_{y \in S} \snorm[\fs{F},n]{tf - y} = \inf_{\{y_i\} \subseteq [0, \infty)} |tf_n - y_n|.
	\]
	To calculate this infimum, we consider two cases for the component $f_n$\textup{:}
	\begin{enumerate}
		\item If $f_n \geq 0$, then for $t > 0$, we have $tf_n \geq 0$. We can choose the sequence $y \in S$ such that its $n$-th component is $y_n = tf_n$. In this case, the distance is $|tf_n - tf_n| = 0$.
		
		\item If $f_n < 0$, then $tf_n < 0$. The closest non-negative number $y_n$ to $tf_n$ is $0$. Thus, the distance is $|tf_n - 0| = |tf_n| = -t f_n$.
	\end{enumerate}
	Combining these cases, we have $	\mathrm{d}_{\fs{F},n}(tf, S) = t \max(0, -f_n)$. Hence,
	\[
	\lim_{t \to 0^+} t^{-1} 	\mathrm{d}_{\fs{F},n}(tf, S) = \max(0, -f_n).
	\]
	For $f$ to be in the cone $\mathrm{T}_sS$, this limit must be $0$ for all $n \in \mathbb{N}$. The condition $\max(0, -f_n) = 0$  means $f_n \geq 0$. Since this must hold for all $n$, we conclude that $f$ is in $\mathrm{T}_sS$ if and only if $f_i \geq 0$ for all $i$. This is precisely the definition of the set $S$. Therefore, for the set of non-negative sequences at the origin, the adjacent cone is the set itself, i.e., $\mathrm{T}_0S = S$.
\end{Exmp}
We now naturally extend this idea to second-order adjacent tangency. This type of tangency was defined for Banach spaces in \cite{pa}.

\begin{Defn}\label{def.1}
	Let \( \emptyset \neq S \subset \fs{F}\), $ s \in S $, and $ e \in \fs{F} $. If there is some $ f \in \fs{F} $ such that
	\begin{equation}\label{eq:hp}
		\forall n \in \NN, \quad \lim_{t \to 0^+} t^{-2}\mathrm{d}_{\fs{F},n} \left(\left(s + t f + \tfrac{1}{2}t^2e \right), S\right)=0,
	\end{equation}
	then $ e $ is called a second-order adjacent tangent vector to \(S\) at $ s $, and we say that  $ f $ is associated with $ e $. The set of all second-order adjacent tangent vectors to \(S\) at \(s\) is denoted by \( \mathrm{T}_s^2S \).
\end{Defn}

\begin{Rem}
	If \( e \in \mathrm{T}_s^2S \) and \( f \) is its associated direction, it follows directly from the definition of \( \mathrm{T}_s^2S \) that \( f \in \mathrm{T}_sS\). Moreover, the zero vector \( \zero{\fs{F}} \) belongs to \( \mathrm{T}_sS\), as any direction can be associated with it.  
	To show that \( \mathrm{T}_s^2S \) is a cone, let \( e \in \mathrm{T}_s^2S \) with associated direction \( f \). For any positive scalar \( r \), consider the vector \( re \). By scaling \( f \) by \( r^{1/2} \), we obtain a new direction \( r^{1/2} f \) that satisfies the conditions for \( re \) to belong to \( \mathrm{T}_s^2S \). Hence, \( \mathrm{T}_s^2S \) is a cone.  
\end{Rem}
\begin{Rem}\label{rem:cone_vs_space}
As we will see in Lemma~\ref{lem:l3}, for a $C^2$-submanifold \(S\), the second-order tangent space $\mathrm{T}_{(s,f)}(\mathrm{T}S)$ consists precisely of those pairs $(f, e)$ whose acceleration component $e$ belongs to the  cone $\mathrm{T}_s^2S$ (with $f$ as the associated velocity). This relationship is clear for an {open subset} $S \subset \fs{F}$. For any $s \in S$, the set of admissible accelerations $\mathrm{T}_s^2S$ is the entire space $\fs{F}$. Consequently, the second-order tangent space $\mathrm{T}_{(s,f)}(\mathrm{T}S)$ is the set of all pairs $(f,e)$ where $f \in \mathrm{T}_sS = \fs{F}$ and $e \in \mathrm{T}_s^2S = \fs{F}$. That is, the space of all pairs is $\fs{F} \times \fs{F}$.
\end{Rem}
\begin{Rem}
	Alternatively, in Definitions \ref{def:2} and \ref{def.1}, we could use the  metric  
	\begin{equation} \label{metric}
		\mt_{\fs{F}}(x, y) = \sum_{n=1}^{\infty} \frac{1}{2^n} \frac{\snorm[\fs{F},n]{x - y}}{1 + \snorm[\fs{F},n]{x - y}}
	\end{equation}  
	which induces the same topology on \( \fs{F} \) as the sequence of seminorms.  
	This equivalence holds because \( \mt_{\fs{F}}(\cdot, S) \to 0 \) if and only if \( \mathrm{d}_{\fs{F},n}(\cdot, S) \to 0 \) for all positive integers \( n \). In other words, both \( \mt_{\fs{F}} \) and the sequence 
	\( (\mathrm{d}_{\fs{F},n}) \) yield the same conclusions about convergence to the set \( S \).  
	
	However,   \fr  spaces lack a canonical metric; multiple metrics induce the same topology and different distances. Seminorms offer a more flexible and practical framework by directly reflecting the underlying topology. 
	
\end{Rem}
Next, we provide natural and straightforward extensions of adjacent and second-order adjacent cones to \fr manifolds, analogous to the Banach manifolds case (see \cite{mon2, mon1}).
\begin{Defn} \label{def:sec2}
	Let $  S \subset \fs{M}$, \(s \in S\). A vector $v \in  \mathrm{T}_sS$ is called an adjacent tangent vector to $S$ at $s$ if there exists a chart $(U, \va)$ around $s$ such that
	\begin{equation}\label{eq:32}
		\forall n \in \NN, \quad \lim_{t \to 0^+} t^{-1}\mathrm{d}_{\fs{F},n} \Big( \va(s) + 
		t \dd \va(s)(v), \va (U \cap S)\Big)=0.
	\end{equation}
	The set of all such \( v \) is denoted by
	\( \mathrm{T}_sS\).
\end{Defn}
\begin{lem}\label{lem:2}
	The set \( \mathrm{T}_sS\) defined in Definition \ref{def:sec2}  is independent of the choice of chart.
\end{lem}
\begin{proof}
	Let $   S \subset \fs{M}$, \(s \in S\), and $ v \in \mathrm{T}_sS$. Let  $(U, \va)$ and $(V, \psi)$ be two charts around $s$. Assume
	Equation \eqref{eq:32} holds for  $(U, \va)$. We show it holds for $(V, \psi)$. 
	
	Since Equation \eqref{eq:32} holds for $(U, \va)$, there is a family of functions \(h_n(t)\colon(0,\epsilon) \to 
	\va (U \cap S)\) for each \(n \in \NN\),
	such that
	\begin{equation*}
		\forall n \in \NN, \quad \lim_{t \to 0^+} t^{-1}\mathrm{d}_{\fs{F},n} 
		\Big( \va(s) + t \dd \va(s)(v), h_n(t)\Big)=0.
	\end{equation*}
	Define   $ \bl{h}_n(t)= -t^{-1}\big(\va(s) + t \dd \va(s)(v)-h_n(t)\big)$ on \((0,\epsilon)\). Then, \(\lim_{t \to 0^+} \bl{h}_n(t)=0\) in all seminorms, and for small \(t\),
	we have
	\[
	\va(s) + t \big( \dd \va(s)(v) + \bl{h}_n(t) \big) \in \va(U \cap S).
	\]
	
	Let \( \phi = \psi \circ \va^{-1} \)  be the transition map.
	By the chain rule, 
	\(
	\dd \psi(s) = \dd \phi (\va(s)) \big( \dd \va(s) \big).
	\)
	Consider the Taylor expansion (Proposition I.2.3, \cite{neeb}) of  $ \phi $ around \(\va(s)\) up to first order
	\[
	\phi(x) = \psi(s) + \dd \phi ({\va(s)})(x - \va(s))  +  \mathrm{R}_1\phi(x)
	\]
	where  $\mathrm{R}_1\phi(x)$ is the first-order remainder.
	Substituting
	\(x = \va(s) + t(\dd{\va(s)}(v) + \bl{h}_n(t)) \) into the Taylor expression yields
	\[
	\phi \big( \va(s) + t \big( \dd \va(s)(v) + \bl{h}_n(t) \big) \big) = \psi(s) + t \big( \dd \psi(s)(v) + \dd \phi (\va(s))(\bl{h}_n(t)) \big) + \mathrm{R}_1\phi(x).
	\]
	Let \( \bl{k}_n(t) = \dd \phi (\va(s))(\bl{h}_n(t))+t^{-1}\mathrm{R}_1\phi(x) \) on \((0,\varepsilon)\), where \( 0 <\varepsilon \leq \epsilon\) is sufficiently small. Since \(\lim_{t \to 0^+}\bl{k}_n(t)  \to 0\) (for all seminorms), for sufficiently small \(t\) we have
	\[
	\psi(s) + t \big( \dd \psi(s)(v) + \bl{k}_n(t) \big) \in \psi(V \cap S).
	\]
	Thus,
	\[
	\forall n \in \NN, \quad \lim_{t \to 0^+} t^{-1} \mathrm{d}_{\fs{F},n} \Big( \psi(s) + t \dd \psi(s)(v), k_n(t) \Big) = 0,
	\]
	where \(k_n(t) = t\bl{k}_n(t)+\psi(s) + t \dd \psi(s)(v)\) on \((0,\varepsilon)\). This implies  Equation \eqref{eq:32} holds for $(V, \psi)$.
\end{proof}
The set \( \mathrm{T}_s S \) is a closed cone in \( \mathrm{T}_s M \). This follows directly from the seminorm condition in Definition~\ref{def:sec2}, as limits and positive scaling preserve the structure. 
For \( C^r \)-submanifolds, adjacent tangent vectors coincide with  tangent vectors. While this result is analogous to the Banach manifold case~\cite{mon2}, we outline the proof in the \fr setting for completeness.

Suppose $ \fs{F}_1 $ is a closed subspace of the \fr space $ \fs{F} $ that splits it. Let $ \fs{F}_2 $ be a topological complement, such that $ \fs{F} = \fs{F}_1 \oplus \fs{F}_2 $. A subset $ {S} \subset \fs{M} $ is called a (split) $ C^r $-\fr submanifold modeled on $ \fs{F}_1 $, for $ 1 \leq r \leq k $, if for any $ p \in {S} $ there exists a $ C^r $-diffeomorphism $ \varphi \colon U \to V$, where $ U \ni p $ is open in $\fs{M}$ and $V$ is an open subset of $\fs{F}$. The set $V$ is required to be a product neighborhood of the form $V = W \times O$, where $W \opn \fs{F}_1$ and $O \opn \fs{F}_2$. The map must then satisfy
\[
\varphi ({S} \cap U) = W \times \set{\zero{\fs{F}_2}}.
\]
Then ${S}$ is a $C^r$-\fr manifold modeled on $\fs{F}_1$, with the maximal $C^r$-atlas including the mappings $\phi|_{{U \cap {S}}}\colon {U \cap {S}} \to {V \cap {S}}$ for all $\varphi$ as described above.

Suppose \(v \in \mathrm{T}_s\fs{M}\) is an adjacent vector to \(S\) at \(s \in S\).
By Lemma \ref{lem:2}, there exists a submanifold chart \((U,\va)\) around \(s\) such that for some open set \( W \opn \fs{F}_1\), we have
$
\va(U \cap S) =  W \times \set{\zero{\fs{F}_2}}
$. 
By Definition \ref{def:sec2}, the element \(s\) satisfies \eqref{eq:32} if and only if there exists a family of functions \(h_n(t)\colon(0,\epsilon) \to \va (U \cap S)\)
such that
\begin{equation*}
	\forall n \in \NN, \quad \lim_{t \to 0^+} t^{-1}\mathrm{d}_{\fs{F},n} \Big(\va(s) + t \dd \va(s)(v), h_n(t)\Big)=0.
\end{equation*}
Define \(\bl{h}_n(t)\) \(= -t^{-1}\big(\va(s) + t \dd \va(s)(v)-h_n(t)\big)\) on \((0,\epsilon)\). Then, \(\lim_{t \to 0^+} \bl{h}_n(t)=0\) in all seminorms, and for small \(t\), we have
\[
\forall n \in \NN, \quad \dd \va(s)(v)+\bl{h}_n(t) \in \fs{F}_1 \times \set{\zero{\fs{F}_2}}.
\]
Since \(\fs{F_1}\) is closed and each  \(\dd\va(s)(v) + \bl{h}_n(t)\) lies in $\fs{F}_1$, taking the limit \(t \to 0^{+}\) yields \(\dd \va(s)(v) \in \fs{F_1}\). Hence, \(v\) is a tangent vector to \(S\) at \(s\).

Conversely, let \( v \in \mathrm{T}_sS\) be a  tangent vector, and \( (U, \varphi) \) a submanifold chart.
By definition of the tangent space, the curve
\(
t \mapsto \varphi(s) + t \, \dd\varphi(s)(v)
\)
lies entirely in \( \varphi(U \cap S) \) for small \( t \). Consequently, \( s \) satisfies \eqref{eq:32}, and hence \( v \) is an adjacent tangent vector to \( S \) at \( s \).

In the following definition and lemma, we will use the notation introduced in \eqref{eq:n1} and \eqref{eq:n2}.

\begin{Defn} \label{def:sec}
	Let \( S \subset \fs{M} \), \( s \in S \), and \( v \in \mathrm{T}_sS\). A vector \( w \in \TS{v}(\TM{M}) \) is called a second-order adjacent tangent vector to \( S \) at \( s \) \textup{(}associated with \( v \)\textup{)} if there exists a chart \( (U, \varphi) \) about \( s \) such that the following two conditions hold:
\begin{enumerate}[label=\textup{(\roman*)}]
	\item $w_{\va_*,1} = v_{\va}$,
	\item For all $n \in \NN$,
	\begin{equation}\label{eq:3}
		\lim_{t \to 0^+} 
		t^{-2}\,\mathrm{d}_{\fs{F},n}\!\left(
		\va(s) + t v_{\va} + \tfrac{1}{2}t^2 w_{\va_*,2},\,
		\va(U \cap S)
		\right) = 0 .
	\end{equation}
\end{enumerate}

	Here, $v_{\va} \coloneqq \dd \varphi_s(v)$ is the local representation of $v$, while $w_{\va_*,1}$ and $w_{\va_*,2}$ are the components of the local representation of $w$, given by $w_{\va_*} \coloneqq \dd {(\va_*)}_v (w) = (w_{\va_*,1}, w_{\va_*,2})$. 
	The set of all such vectors \( w \) is denoted by \( \mathrm{T}_s^2S \).
\end{Defn}
\begin{lem}\label{lem:1}
	The definition of \( \mathrm{T}_s^2S \) in	Definition \ref{def:sec} is independent of the choice of chart.	
\end{lem}
\begin{proof}
	Let \( S \subset \fs{M} \), \( s \in S \), and \( v \in \mathrm{T}_sS\). Consider two charts \( (U, \varphi) \) and \( (V, \psi) \) around \( s \), and let \( \phi = \psi \circ \varphi^{-1} \) be the transition map. Suppose \( w \in \TS{v}(\TM{M}) \) satisfies $w_{\va_*,1} = v_{\va}$ and \eqref{eq:3} holds in \( (U, \varphi) \). We will show that \eqref{eq:3} also holds in \( (V, \psi) \). 
	
	Equation  \eqref{eq:3} holds if and only if there exists a family of function \(h_n(t)\colon(0,\epsilon) \to \va (U \cap S)\)
	such that
	\[
	\forall n \in \NN, \quad \lim_{t \to 0^+} t^{-2}\mathrm{d}_{\fs{F},n} 
	\Big(\big(\va(s) + t v_{\va} + \tfrac{1}{2}t^2w_{\va_*,2} \big), h_n(t) \Big)=0.
	\]
	Define  \(\bl{h}_n(t) \coloneqq -t^{-2}\big(\va(s) + t v_{\va} + \tfrac{1}{2}t^2w_{\va_*,2}-h(t)\big)\) on \((0,\epsilon)\). Then, \(\lim_{t \to 0^+} \bl{h}_n(t)=0\) in all seminorms, and for small \(t\),
	we have
	\[
	k_n(t) \coloneqq \va(s) + t v_{\va} + \tfrac{1}{2}t^2 \big(w_{\va_*,2} + \bl{h}_n(t) \big) \in \va(U \cap S).
	\]
	Without loss of generality, we choose \(\epsilon\) sufficiently small so that \(\phi(k_n(t)) \in \psi (U \cap V \cap S)\). 
	
	We aim to show that  
	\begin{equation}\label{eq:new}
		\forall n \in \NN, \quad \lim_{t \to 0^+} t^{-2}\mathrm{d}_{\fs{F},n} \Big(
		\big(\psi(s) + tv_{\psi} + \tfrac{1}{2}t^2w_{\psi_*,2}\big), \psi(V \cap S)\Big) = 0.
	\end{equation}
	To this end, we will express the terms in the limit condition using the chart \((V,\psi)\),	
	based on the given relationships between \( v_{\va} \), \( v_{\psi} \), \( w_{\va_*,2} \), and \( w_{\psi_*,2} \),
	and \(\phi\), namely
	\begin{equation}\label{eq:4}
		v_{\psi} = \dd \phi_{\va(s)}(v_{\va}),\, w_{\psi*,1} = v_{\va} , \text{and } w_{\psi*,2} = \dd^2 \phi_{\va(s)}(v_{\va}, w_{\va_*,1} )+\dd \phi_{\va(s)}( w_{\va_*,2}) .
	\end{equation}
	Using the Taylor expansion up to second order of \(\phi\) around \(\va(s)\), we have
	\[
	\phi(x) = \psi(s) + \dd \phi_{\va(s)}(x - \va(s)) + \tfrac{1}{2}\dd^2\phi_{\va(s)}\big(x - \va(s), x - \va(s)\big) + \mathrm{R}_2\phi(x)
	\]
	where  $\mathrm{R}_2\phi(x)$ is the second-order remainder.
	Substituting
	\(x = k_n(t)\) into the Taylor expansion results in
	\[
	\phi\big(k_n(t)\big) = \psi(s) + \dd \phi_{\va(s)} \big( k_n(t) - \va(s)\big) 
	+ \tfrac{1}{2}\dd^2\phi_{\va(s)} \big( k_n(t)- \va(s), k_n(t)- \va(s) \big) + \mathrm{R}_2\phi(x).
	\]
	Applying the expressions in \eqref{eq:4}  and substituting  \(k_n(t)\) into the later equation yields 
	\begin{equation}\label{eq:new2}
		\phi\big(k_n(t)\big) = \psi(s) + tv_{\psi}
		+ \tfrac{1}{2} t^2 \big( w_{\psi_*,2} \big) + \mathrm{R}_2\phi(x).
	\end{equation}
	Since \(\phi(k_n(t)) \in \psi ( V \cap S)\), for sufficiently small $t>0$, there exists a $h_n(t) \in S$ such that \( \phi(k_n(t))= \psi(h_n(t))\). Thus, Equation \eqref{eq:new2} implies
	\[
	\forall n \in \NN, \quad \lim_{t \to 0^+} t^{-2}\mathrm{d}_{\fs{F},n} \Big(
	\big(\psi(s) + tv_{\psi} + \tfrac{1}{2}t^2w_{\psi_*,2}\big), \psi(h_n(t))\Big) =  \lim_{t \to 0^+}t^{-2}\mathrm{R}_2\phi(x)=0.
	\]
	Since \(\psi(h_n(t)) \in \psi ( V \cap S)\), it follows that
	the pseudo-distances to the set \(\psi(V \cap S)\) is at most the pseudo-distances to the specific point \(\psi(h_n(t))\), i.e.,
	\[
	\mathrm{d}_{\fs{F},n} \Big( \psi(s) + t v_{\psi} + \tfrac{1}{2} t^2 w_{\psi_*,2}, \psi(V \cap S) \Big) 
	\leq 
	\mathrm{d}_{\fs{F},n} \Big( \psi(s) + t v_{\psi} + \tfrac{1}{2} t^2 w_{\psi_*,2}, \psi(h_n(t)) \Big).
	\]
	Thus, Equation \eqref{eq:new} holds true.
\end{proof}
\begin{Rem}
	If the manifold $\fs{M}$ coincides with its model space $\fs{F}$, then 
	Definition~\ref{def:sec} reduces to Definition~\ref{def.1}. To see this, the arguments analogous to those in the Banach case can be applied; see~\cite[Remark 2.12]{mon1}.
\end{Rem}
The proof of the following result relies primarily on the properties of 
submanifold charts and on limit arguments, which can be adapted from 
Banach manifolds (see~\cite[Theorem~2.13]{mon1}) to our context with 
minor modifications.
\begin{lem}\label{lem:l3}
	Let $S$ be a $C^2$-submanifold of $\fs{M}$ modeled on $\fs{F}_1$, and let $s \in S$.  
	Then the following statements are equivalent:
	\begin{enumerate}[label=\textup{(\roman*)}]
		\item $w \in \mathrm{T}^2_s S$ with associated tangent vector $v \in \mathrm{T}_s S$;
		\item $w \in \TS{v}(\mathrm{T}S)$ for some $v \in \mathrm{T}_s S$.
	\end{enumerate}
\end{lem}
\begin{proof}
	Suppose \(w \in \mathrm{T}^2_sS\) and \(v\) is its associated vector. Then \(\TS{v}(\mathrm{T}S)\)
	is the tangent space  at \(v\) to \(\mathrm{T}S\). By Lemma \ref{lem:1}, there exists a submanifold chart \((U,\va)\) at \(s\) for \(S\), such that  for some  \( W \opn \fs{F_1}\), we have
	\begin{equation}\label{eq:f2}
		\va(U \cap S) = W \times \set{\zero{\fs{F}_2}},
	\end{equation}
	where \(\fs{F}_2\) is a complement of \( \fs{F}_1 \).
	The condition
	\begin{equation*}
		\forall n \in \NN, \quad \lim_{t \to 0^+} t^{-2}\mathrm{d}_{\fs{F},n} \Big(\big(\va(s) + t v_{\va} + \tfrac{1}{2}t^2w_{\va_*,2} \big), \va (U \cap S)\Big)=0
	\end{equation*}
	is valid if and only if  there exists a family of functions \(h_n(t)\colon(0,\epsilon) \to \va(U \cap S)\)  such that
	\begin{equation*}
		\forall n \in \NN, \quad \lim_{t \to 0^+} t^{-2}\mathrm{d}_{\fs{F},n} \Big(\big(\va(s) + t v_{\va} + \tfrac{1}{2}t^2w_{\va_*,2} \big), h_n(t)\Big)=0.
	\end{equation*}
	Define 
	\begin{equation*}
		\bl{h}_n(t) \coloneqq	-t^{-2} \Big(\big(\va(s) + t v_{\va} + \tfrac{1}{2}t^2w_{\va_*,2} \big) - h_n(t)\Big).
	\end{equation*}
	Therefore, 
	\(\lim_{t \to 0^+} \bl{h}_n(t) =0\), in all seminorms. 
	Moreover, for small \(t\), we have
	\begin{equation}\label{eq:fl3}
		\va(s) + t v_{\va} + \tfrac{1}{2}t^2(w_{\va_*,2} + \bl{h}_n(t)) \in \fs{F}_1.
	\end{equation}
	Since $ \mathrm{T}_sS$ is the tangent space at \(s\) to \(S\) and \(v \in \mathrm{T}_sS\), it follows that \(v_{\va} \in \fs{F}_1\).
	
	Furthermore, from Equation \eqref{eq:fl3}, for all \(n \in \nn\) we have
	\[
	w_{\va_*,2} + \bl{h}_n(t) \in \fs{F}_1, \quad \forall t > 0.
	\]
	Taking the limit as  \(t \to 0^+\), we deduce  that \(w_{\va_*,2} \in \fs{F}_1\). Since \(w_{\va_*,1} = v_{\va} \in \fs{F}_1\), it follows that \( w_{\va_*}=(w_{\va_*,1}, w_{\va_*,2}) \in \fs{F}_1 \times \fs{F}_1 \), which implies that \(w \in \TS{v}(TS)\).
	
	Conversely, let $w \in \TS{v}(TS)$. Since $S$ is a $C^2$-submanifold of $\fs{M}$ modeled on $\fs{F}_1$, its tangent bundle $\mathrm{T}S$ is a $C^1$-manifold modeled on the product space $\fs{F}_1 \times \fs{F}_1$. The tangent space at any point of $\mathrm{T}S$ is therefore also modeled on $\fs{F}_1 \times \fs{F}_1$. This implies that for any submanifold chart $(U,\va)$ at $s$ for $S$, the local components of the vector $w$ must satisfy $w_{\va_*,1} \in \fs{F}_1$ and $w_{\va_*,2} \in \fs{F}_1$.
	
	Now, let $t>0$ be small enough such that $\va(s) + t v_{\va} + \tfrac{1}{2}t^2w_{\va_*,2} \in W$. Because $v_\va \in \fs{F}_1$ and $w_{\va_*,2} \in \fs{F}_1$, their linear combination also lies in $\fs{F}_1$. Thus, we have
	\begin{equation*}
		\big(\va(s) + t v_{\va} + \tfrac{1}{2}t^2w_{\va_*,2}\big) \in \va(U \cap S) = W \times \set{\zero{\fs{F}_2}}.
	\end{equation*}
	This directly implies that the limit condition is satisfied:
	\begin{equation*}
		\forall n \in \NN, \quad \lim_{t \to 0^+} t^{-2}\mathrm{d}_{\fs{F},n} \Big(\big(\va(s) + t v_{\va} + \tfrac{1}{2}t^2w_{\va_*,2} \big), \va (U \cap S)\Big)=0.
	\end{equation*}
	Therefore, $w \in \mathrm{T}^2_sS$, which completes the proof.
\end{proof}
Having established the necessary tools for studying spray-invariant sets, we now introduce a specific set that plays a crucial role. 
\begin{Defn} \label{def:admiset}
	Let \( \s \) be a spray on \(\fs{M}\), and \(S \subset \fs{M}\) a non-empty subset. A tangent vector
	\(v \in \TM{M}\) is called a 
	\((\mathrm{T}^2S,\s)\)-admissible vector if
	\[
	\tau(v) \in S \quad \text{and }\quad \s(v) \in \mathrm{T}_{\scriptscriptstyle \tau(v)}^2 S.
	\]
	The set of such vectors, denoted by  \(A_{\s, S} \), is called the 	\((\mathrm{T}^2S,\s)\)-admissible set for \(\s\) and \(S\).
\end{Defn}
By directly applying  Definition \ref{def:sec} and Lemma \ref{lem:1}, we obtain a local description of the set \(A_{\s, S} \). Let $ v \in  A_{\s, S} $. Then, there exists a chart
$ \va\colon U \to \fs{F} $ at $ \bl{v}\coloneqq \tau(v)  $ such that
\begin{equation}\label{eq:5g}
	\forall n \in \NN, \quad \lim_{t \to 0^+} t^{-2}\mathrm{d}_{\fs{F},n} \Big(\big(\va(\bl{v} ) + t v_{\va} + \tfrac{1}{2}t^2{\s_{(\va_*,2)}(v)} \big), \va (U \cap S)\Big)=0,
\end{equation}
where, in coordinates \( \varphi_*\colon\mathrm{T}U \to \fs{F} \times \fs{F} \), the spray decomposes as follows
\begin{equation}\label{eq:6}
	\s_{(\va_*)}(v) = \dd (\va_*)_{\bl{v}}\s(v) = \Big(\mathrm{Pr}_1({\s_{(\va_*)}(v)})=v_{\va}, \mathrm{Pr}_2({\s_{(\va_*)}(v)}) \eqqcolon {\s_{(\va_*,2)}(v)} \Big)  \in \fs{F}\times\fs{F}.
\end{equation}

\begin{Rem}\label{rem:1}
	Let \( \s\) and \( \overline{\s} \) be projectively equivalent sprays, i.e., \(\ol{\s} \in [\s]\). 	
	In general, the admissible sets \( A_{\s, S}  \) and \( {A}_{{\s}, \ol{\s}} \) need not coincide. From the projective relation \eqref{eq:gd4}, locally
	\begin{equation*}
		{{\s}}_{\va,2} (x,v_{\va}) 
		= {\ol{\s}}_{\va,2} (x,v_{\va}) + P(x,v_{\va})v_{\va}, \quad \text{ for } v \in \mathrm{T}_x\fs{M}.
	\end{equation*}
	Since \( \mathrm{T}_{\tau(v)}^2 S \) is generally only a closed cone \textup{(}not a linear space\textup{)}, the term \( P(x, v_\varphi)v_\varphi \) may result in \( \s(v) \notin \mathrm{T}_{\tau(v)}^2 S \) even if \( \overline{\s}(v) \in \mathrm{T}_{\tau(v)}^2 S \). Thus, \( {A}_{\s, S} \) is not preserved under projective equivalence.
	However, if \( S \) is a \( C^2 \)-submanifold, then by Lemma \ref{lem:l3} we have
	\[
	\overline{\s}(v) \in \mathrm{T}_{\tau(v)}^2 S \iff \overline{\s}(v) \in \mathrm{T}_w(\mathrm{T}S) \text{ for some } w \in \mathrm{T}_{\tau(v)}S,
	\]
	where \( \mathrm{T}_w(\mathrm{T}S) \) is a linear subspace of \( \mathrm{T}(\TM{M}) \). Since \( P(x, v_\varphi)v_\varphi \in  \mathrm{T}_{\tau(v)}S \), it follows that
	\[
	\s(v) = \overline{\s}(v) + P(x, v_\varphi)v_\varphi \in \mathrm{T}_{\tau(v)}^2 S.
	\]
	Therefore, 
	\( A_{\s, S}  \) and \( {A}_{{\s}, \ol{\s}} \) are  the same in this case.
\end{Rem}

\begin{theorem}\label{th:1}
	Let \( \s \) be a spray on \(\fs{M}\), $g\colon I \subset \RR \rightarrow \fs{M}$   its geodesic, and \(S \subset \fs{M}\) a non-empty closed subset. Then, for all \(t \) in $ I $, \(g(t) \in S\)
	if and only if \(g'(t)\in A_{\s, S} \).
\end{theorem}
\begin{proof}
	Assume that \(t \in I \) and \(g(t) \in S\). Let $ \epsilon > 0 $ be  sufficiently
	small such that $ t+s \in I $ and \(g(t+s) \in S\) for all $ s \in (0,\epsilon] $.
	Let \(\va\colon U \to \fs{F}\) be a chart around \(g(t)\). 
	Using the properties of charts, we can express \(g'(t)\) in terms of the chart coordinates and their derivatives as follows
	\begin{align*}
		(g'(t))_{\va}&=\dd \va \big( g(t) \big)(g'(t))=
		(\va \circ g)'(t).
	\end{align*}
	Therefore, by \eqref{eq:6},  we get
	\begin{align*}
		\s_{(\va_*)}(g'(t))&=\dd (\va_*) \big( g(t)\big) (\s (g'(t))=
		\dd (\va_*) \big(g(t)\big)(g'(t))\\
		&=(\va_*({g')})'(t)\\
		&=\big( (\va \circ g)'(t),(\va \circ g)''(t)\big).
	\end{align*}
	Thus, for sufficiently small \(s\), we have  
	\begin{multline*}
		\forall n \in \NN, \quad  s^{-2}\mathrm{d}_{\fs{F},n} \Big(\big(\va ( g(t)) +
		s(g'(t))_{\va} + \tfrac{1}{2}s^2{\s_{(\va_*,2)}(g'(t))} \big), \va (U \cap S)\Big)\leq \\
		\leq s^{-2}\mathrm{d}_{\fs{F},n} \Big(\big(\va ( g(t)) +
		s \big( \va \circ g \big)'(t)  + \tfrac{1}{2}s^2 \big( \va \circ g)\big)''(t) \big), 
		\va ( g(t+s)) \Big).
	\end{multline*}
	Since \(g(t)\) is \(C^2\), the right-hand side vanishes as $ s \to 0^+ $. Therefore, \(g'(t)\in A_{\s, S}\).
	
	Now, assume that \(t \in I\) and \(g'(t) \in A_{\s, S}\). Then,
	\begin{equation}\label{eq:7}
		\forall n \in \NN, \quad \lim_{\delta \to 0^+} \delta^{-2} \mathrm{d}_{\fs{F},n} \Big( \va( g(t)) + \delta(g'(t))_{\va} + \tfrac{1}{2} \delta^2 \s_{(\va_*,2)}(g'(t)), \, \va(U \cap S) \Big) = 0,
	\end{equation}
	which characterizes the admissibility of \(g'(t)\) relative to the set \(S\).
	This condition holds if and only if there exists a family of functions \(h_n(\delta)\colon (0,\epsilon) \to \va(U \cap S)\) such that
	\[
	\forall n \in \NN, \quad \lim_{\delta \to 0^+} \delta^{-2} \mathrm{d}_{\fs{F},n} \Big( \va(g(t)) + \delta(g'(t))_{\va} + \tfrac{1}{2} \delta^2 \s_{(\va_*,2)}(g'(t)), h_n(\delta) \Big) = 0.
	\]
	Define  
	\[
	\bl{h}_n(\delta) := -\delta^2 \left( \va(g(t)) + \delta(g'(t))_{\va} + \tfrac{1}{2} \delta^2 \s_{(\va_*,2)}(g'(t)) - h_n(\delta) \right), \quad \delta \in (0, \epsilon).
	\]
	Then \(\lim_{\delta \to 0^+} \bl{h}_n(\delta) = 0\) in all seminorms. Consequently,
	\[
	\va(g(t)) + \delta (g'(t))_{\va} + \tfrac{1}{2} \delta^2 \left( \s_{(\va_*,2)}(g'(t)) + \bl{h}_n(\delta) \right) \in \va(U \cap S), \quad \text{for } \delta \in (0, \epsilon).
	\]
	Since \(\va(U \cap S)\) is closed in \(\va(U)\), taking the limit as \(\delta \to 0^+\) yields
	\[
	\va(g(t)) \in \va(U \cap S),
	\]
	and therefore \(g(t) \in S\).
\end{proof}

We can now introduce the concept of a spray-invariant set with respect to a spray.

\begin{Defn}\label{def:fis}
	Let \(\s\) be a spray on \(\fs{M}\), and let \(S\) be a subset of \(\fs{M}\) such that \(A_{\s, S}\) is not empty. We say \(S\) is spray-invariant with respect to \(\s\) if, for any geodesic \(g \colon I \to \fs{M}\) of \(\s\) such that \( 0 \in I \), \(g(0) \in S\), and \(g'(0) \in A_{\s, S}\), then \(g(t) \in S\) for all \(t \in I\).
\end{Defn}
By Theorem \ref{th:1}, a closed subset \(S \subset \fs{M}\) is spray-invariant if, for any geodesic \(g \colon I \to \fs{M}\) of \(\s\) such that \( 0 \in I \), \(g(0) \in S\), and \(g'(0) \in A_{\s, S}\), then \(g'(t) \in A_{\s, S}\) for all \(t \in I\).
\begin{Exmp}\label{ex:1}
	Let \( \fs{E} = C^\infty(\RR,\RR)\) be the \fr space of smooth real-valued functions on \(\RR\) whose topology is defined by a family of seminorms 
	\[
	\snorm[\fs{E},n]{f} \coloneqq \sup_{x \in [-j, j]} |f^{(m)}(x)|
	\]
	Here, $f^{(m)}$ denotes the $m$-th derivative of the function $f$,  $j \in \mathbb{N}$, and $m \in \{0, 1, 2, \dots\}$. This space is a \fr manifold modeled on itself with the tangent bundle \( \mathrm{T}\fs{E}\cong \fs{E} \times \fs{E} \).
	
	Consider a flat spray \( \s(f, v) = (f, v, v, \zero{\fs{E}}) \), where geodesics are affine paths \( \gamma(t) = f + t v \). Define the subset \( \mathsf{S} = \mathsf{S}_{+} \cup \mathsf{S}_{-}, \) where
	\[
	\mathsf{S}_{+} \coloneqq  \{ f \in \fs{E} \mid \supp(f) \subseteq [0, \infty) \}, \quad \mathsf{S}_{-} \coloneqq  \{ f \in \fs{E} \mid \supp(f) \subseteq (-\infty, 0] \}.
	\]
	The set $S$ is the union of two closed subspaces $\mathsf{S}_{+} $ and $\mathsf{S}_{-}$ which are smooth submanifolds of \(\fs{E}\). However, it fails to be a manifold because there exists no neighborhood of the zero function in \( \mathsf{S} \) that is locally homeomorphic to a linear subspace.
	Consider any neighborhood $N$ of the zero function in $S$. This neighborhood will contain functions from $S_+$  and functions from $S_-$. In a linear space, one can always find a continuous path between any two nearby points. However, any continuous path in $\fs{E}$ from a function in $S_+$ to a function in $S_-$ must pass through functions whose support lies on both sides of zero. Such functions are not in $S_+$ and not in $S_-$, and therefore are not in $S$. The only point connecting these two sets is the zero function itself. Consequently, if we remove the origin from the neighborhood $N$, it splits into two disconnected pieces. This local structure is not homeomorphic to a linear space.

	The adjacent cones to \( \mathsf{S} \) at non-zero points are given by
	\[
	\mathrm{T}_f \mathsf{S}=\mathrm{T}_f \mathsf{S}_{+} = \left\{ v \in \fs{E} \,\middle|\, \supp(v) \subseteq [0, \infty) \right\}, \quad \text{for } f \in \mathsf{S}_{+} \setminus \{0\},
	\]
	\[
	\mathrm{T}_f \mathsf{S}=\mathrm{T}_f \mathsf{S}_{-} = \left\{ v \in \fs{E} \,\middle|\, \supp(v) \subseteq (-\infty, 0] \right\}, \quad \text{for } f \in \mathsf{S}_{-} \setminus \{0\}.
	\]

	At the origin, we prove that the adjacent cone is the union of the two subspaces, i.e., $\mathrm{T}_0 \mathsf{S} = \mathsf{S}_{+} \cup \mathsf{S}_{-}$.

	Let $v \in \mathsf{S}_{+} \cup \mathsf{S}_{-}$.
	To check if $v \in \mathrm{T}_0 \mathsf{S}$, we must show the limit condition holds for every seminorm $\snorm[\fs{E},n]{\cdot}$.  For any $t > 0$, the function $tv$ is also in $\mathsf{S}$. We can therefore choose the point $h = tv \in \mathsf{S}$ to measure the pseudo-distance. For any $n \in \mathbb{N}$, we have
	\[
	\mathrm{d}_{\fs{E},n}\big(0 + tv, \mathsf{S}\big) = \inf_{g \in \mathsf{S}} \snorm[\fs{E},n]{tv - g} \leq \snorm[\fs{E},n]{tv - tv} = 0.
	\]
	Since the pseudo-distance is zero for every $n$, the limit condition is trivially satisfied. Thus, any $v \in \mathsf{S}_{+} \cup \mathsf{S}_{-}$ is in $\mathrm{T}_0 \mathsf{S}$.

	Now, let $v \in \fs{E}$ be a vector such that $v \notin \mathsf{S}_{+} \cup \mathsf{S}_{-}$. This means there exist points $x_1 < 0$ and $x_2 > 0$ such that $v(x_1) \neq 0$ and $v(x_2) \neq 0$. Select an integer $j \in \mathbb{N}$ large enough such that the compact interval $K_j = [-j, j]$ contains both $x_1$ and $x_2$. Let $n_0$ be the index corresponding to the pair $(j, m=0)$. The associated seminorm is $\snorm[\fs{E},n_0]{f} = \sup_{x \in [-j, j]} |f(x)|$.	The function $tv$ has non-zero values on both sides of the origin within $K_j$. Any function $g \in \mathsf{S}$ has support on only one side of the origin.
	\begin{itemize}
		\item If we choose $g \in \mathsf{S}_+$, then $g(x_1)=0$, so $\snorm[\fs{E},n_0]{tv-g} \geq |tv(x_1) - g(x_1)| = t|v(x_1)|$.
		\item If we choose $g \in \mathsf{S}_-$, then $g(x_2)=0$, so $\snorm[\fs{E},n_0]{tv-g} \geq |tv(x_2) - g(x_2)| = t|v(x_2)|$.
	\end{itemize}
	In either case, the infimum pseudo-distance is bounded below. Let $C = \min(|v(x_1)|, |v(x_2)|) > 0$. Then $\mathrm{d}_{\fs{E},n_0}(tv, \mathsf{S}) \geq C t$. The limit for this specific seminorm is therefore bounded below:
	\[
	\lim_{t \to 0^+} t^{-1}\mathrm{d}_{\fs{E},n_0}\big(tv, \mathsf{S}\big) \geq \lim_{t \to 0^+} t^{-1} (C t) = C > 0.
	\]
	Since the limit condition must hold for all $n \in \mathbb{N}$ for a vector to be in $\mathrm{T}_0 \mathsf{S}$, and we have found at least one seminorm \textup{(}indexed by $n_0$\textup{) }for which it fails, $v$ is not in $\mathrm{T}_0 \mathsf{S}$. Combining both inclusions, we have shown that $\mathrm{T}_0 \mathsf{S} = \mathsf{S}_{+} \cup \mathsf{S}_{-}$.

	For \( f \in \mathsf{S}_{+} \setminus \{0\} \) \textup{(}resp.\ {\( \mathsf{S}_{-} \setminus \{0\} \)}\textup{)}, we have
\[
\mathrm{T}_f^2 \mathsf{S} = \mathrm{T}_f \mathsf{S}_{+} \quad \text{\textup{(}resp. } \mathrm{T}_f \mathsf{S}_{-} \textup{)},
\]
since infinitesimal perturbations preserve the support condition.

\noindent
For \( f = 0 \), we have
\[
\mathrm{T}_0^2 \mathsf{S} = \mathsf{S}_{+} \cup \mathsf{S}_{-}, \quad \text{and} \quad \s(v) = \zero{\fs{E}} \in \mathrm{T}_0^2 \mathsf{S}.
\]
The flat spray \( \s \) trivially satisfies \( \s(v) \in \mathrm{T}_f^2 \mathsf{S} \), as \( \zero{\fs{E}} \in \mathrm{T}_f^2 \mathsf{S} \) for all \( f \in \mathsf{S} \). 
Thus, 
\[
A_{\s, \mathsf{S}} = \bigcup_{f \in \mathsf{S}} \left\{ (f, v) \in \mathrm{T} \fs{E} \,\middle|\, v \in \mathrm{T}_f \mathsf{S}_{+} \text{ or } v \in \mathrm{T}_f \mathsf{S}_{-} \right\}.
\]
Let \( f \in \mathsf{S} \) be a point and \( v \in A_{\s, \mathsf{S}} \) be a tangent vector at \( f \). By direct verification, for the geodesic \( \gamma\colon \mathbb{R} \to \mathsf{M} \) of \( \s \) with initial conditions \( \gamma(0) = f \) and \( \gamma'(0) = v\), we have 
\[
\gamma(t) = f + t v \in \mathsf{S}, \quad \forall t \in \mathbb{R}.
\]
Hence, \( \mathsf{S} \) is spray-invariant.
\end{Exmp}
\begin{Rem}
	As noted in Remark \ref{rem:1}, the admissible sets $A_{\s, S}$ and $A_{\overline{\s}, S}$ for projectively equivalent sprays $\s$ and $\overline{\s}$ generally differ when $S$ is not differentiable submanifold. Thus, spray invariance of S with respect to one of these sprays does not imply spray invariance with respect to the other. This implies the sensitivity of geometric properties of singular sets to the specific projective parametrization of sprays.
\end{Rem}
\begin{Exmp}
	Let $\chi_\delta$ and $\chi_\varepsilon$ be standard smooth bump functions. A function $\chi_\delta: \mathbb{R} \to \mathbb{R}$ is a $C^\infty$ function such that $\chi_\delta(x) > 0$ for $x \in (-\delta/2, \delta/2)$ and $\chi_\delta(x) = 0$ otherwise. We choose these functions to be symmetric, i.e., $\chi_\delta(x) = \chi_\delta(-x)$.
	
	Consider a tangent vector $v \in \fs{E} = C^\infty(\RR, \RR)$ defined as $v(x) = \chi_\delta(x)$ for some $0 < \delta \le \varepsilon/2$. The support of $v$ is the compact interval $[-\delta/2, \delta/2]$, which is centered at the origin. Then
	\[
	\alpha_{\epsilon}(v) = \int_{\RR} \chi_\varepsilon(x) \chi_\delta(x) \, dx = \int_{-\delta/2}^{\delta/2} \chi_\varepsilon(x) \chi_\delta(x) \, dx > 0.
	\]
Now, define the spray \( \tilde{\s} \) by
\[
\tilde{\s}(f, v) \coloneqq  (f, v, v, -2\alpha_{\epsilon}(v)\cdot v).
\]
This yields a projectively equivalent spray since it modifies the second derivative by a multiple of the adjacent tangent vector. 

Let \( f \in \mathsf{S}_+ \) \textup{(}i.e., \( \supp(f) \subseteq [0, \infty) \)\textup{)} and the initial tangent be \( \gamma'(0) = v = \chi_\delta \) with \( \delta > 0 \). The support of \( v \) is \( [-\delta/2, \delta/2] \), extending to the negative real line. At \( t = 0 \), \( \alpha(\gamma(0), \gamma'(0)) = \alpha_{\epsilon}(v) > 0 \), so \( \gamma''(0) = -2\alpha_{\epsilon}(v) v \).
The Taylor expansion of the geodesic around \( t = 0 \) is given by
\[
\gamma(t) = f + t v - t^2 \alpha_{\epsilon}(v) v + \mathcal{O}(t^3) = f + t(1 - t\alpha_{\epsilon}( v)) v + \mathcal{O}(t^3).
\]
Since \( \supp(f) \subseteq [0, \infty) \) and \( \supp(v) = [-\delta/2, \delta/2] \) with \( \delta > 0 \), for any \( t > 0 \) \textup{(}even infinitesimally small\textup{)}, the term \( tv \) will introduce a non-zero component to \( \gamma(t) \) with support on \( (-\infty, 0) \), unless \( v \) was identically zero on \( (-\infty, 0) \), which \( \chi_\delta \) is not. Therefore, \( \gamma(t) \) will leave \( \mathsf{S}_+ \), and hence \(S\), for \( t > 0 \).
Similarly, if we start with \( f \in \mathsf{S}_- \) and \( v = \chi_\delta \), the geodesic will leave \( \mathsf{S}_- \), and hence \(S\), for \( t > 0 \).
Thus, while \( \mathsf{S} = \mathsf{S}_+ \cup \mathsf{S}_- \) is invariant under the flat spray, it is not invariant under the projectively equivalent spray \( \tilde{\s} \).
\end{Exmp}
Now, using the concept of admissible sets, we can characterize totally geodesic submanifolds.  
Let $\s$ be a spray on a manifold $\fs{M}$, and let ${S} \subset \fs{M}$ be a submanifold. The submanifold ${S}$ is called \textit{totally geodesic} (with respect to $\s$) if, for all $p \in {S}$ and all $v \in \mathrm{T}_p{S}$, the geodesic $\gamma_v(t)$ in $\fs{M}$ starting at $p$ with initial velocity $v$ satisfies \( \gamma_v(t) \in {S} \) for all $t$.
For a totally geodesic submanifold  \(S\), the restriction   $\s_{S} \coloneqq \s|_{T{S}}$ is a spray on \(S\), and every geodesic of the induced spray $\s_{S}$ is also a geodesic of $\s$ on $\fs{M}$. By definition, totally geodesic submanifolds are spray-invariant.

\begin{theorem}\label{th:sub} Let \(\s\) be a spray on \(\fs{M}\),  and let \(S\) be a \(C^3\)-submanifold of \(\fs{M}\).
	Then \(S\) is totally  geodesic if and only if 	$A_{\s, S}  =\mathrm{T}S$. 
\end{theorem} 

\begin{proof}
	First, we prove that \( A_{\s, S} = \s^{-1}(\mathrm{T}(\mathrm{T}S)) \).  
	Suppose \( v \in A_{\s, S} \). Then \( \tau(v) \in S \) and \( \s(v) \in \mathrm{T}^2_{\scriptscriptstyle \tau(v)} S \), with associated vector \( v \in \mathrm{T}_{\scriptscriptstyle \tau(v)} S \).  
	By Lemma~\ref{lem:l3}, it follows that \( \s(v) \in \mathrm{T}_v(\mathrm{T}S) \). Since  
	\( \mathrm{T}_v(\mathrm{T}S) \subset \mathrm{T}(\mathrm{T}S) \), we conclude that \( \s(v) \in \mathrm{T}(\mathrm{T}S) \).
	
	Conversely, suppose \( \s(v) \in \mathrm{T}(\mathrm{T}S) \). Then \( \s(v) \in \mathrm{T}_v(\mathrm{T}S) \), and hence \( v \in \mathrm{T}_{\scriptscriptstyle \tau(v)} S \).  
	By Lemma~\ref{lem:l3}, this implies \( \s(v) \in \mathrm{T}^2_{\scriptscriptstyle \tau(v)} S \), and thus \( v \in A_{\s, S} \).  
	Therefore, we have
	\begin{equation}\label{eq:lm}
		A_{\s, S} = \s^{-1}(\mathrm{T}(\mathrm{T}S)).
	\end{equation}
	
	This means that \( A_{\s, S} \) consists of all vectors \( v \in \mathrm{T} \fs{M} \) such that \( \s(v) \in \mathrm{T}(\mathrm{T}S) \).  
	In particular, if \( \s(v) \in \mathrm{T}(\mathrm{T}S) \), then the geodesic starting at \( v \) remains in \( \mathrm{T}S \).
	Now assume that \( S \) is totally geodesic. Then for all \( v \in \mathrm{T}S \), the geodesic of \( \s \) starting at \( v \) remains in \( S \), so \( \s(v) \in \mathrm{T}(\mathrm{T}S) \). Hence, by \eqref{eq:lm}, \( A_{\s, S} = \mathrm{T}S \).
	Conversely, if \( A_{\s, S} = \mathrm{T}S \), then  \( S \) is totally geodesic by definition.
	
\end{proof}
\begin{Exmp}\label{ex:2}
	Let \( \mathcal{M} = C^\infty(\mathbb{R}, \mathbb{R}^2) \) be the \fr space of smooth functions from \( \mathbb{R} \) to \( \mathbb{R}^2 \), equipped with the flat spray
	\(
	\s(f, v) = (f, v, v, (0, 0))
	\),	where \( (0, 0) \) denotes the zero function in \( \mathcal{M} \).
	Consider the subset \( S \subseteq \mathcal{M} \) defined by
	\[
	S \coloneqq \{ f \in \mathcal{M} \mid f(x) = (h(x), h(x)^2) \text{ for some } h \in C^\infty(\mathbb{R}, \mathbb{R}) \}.
	\]
	Let \( E = C^\infty(\mathbb{R}, \mathbb{R}) \). Define the map 
	\[
	\Phi\colon E \to \mathcal{M}, \quad \Phi(h)(x) = \big(h(x), h(x)^2\big).
	\]
	Then \( \mathrm{Im}(\Phi) = S \). We first show that \( \Phi \) is a smooth injective immersion. Maps between \fr spaces are Michal-Bastiani smooth  if and only if they are conveniently smooth, i.e., they map smooth curves to smooth curves. 
	Let \( \gamma\colon \mathbb{R} \to E \) be a smooth curve. Then
	\[
	(\Phi \circ \gamma)(t)(x) = \big(\gamma(t)(x), \gamma(t)(x)^2\big),
	\]
	which is smooth in both \( t \) and \( x \), hence \( \Phi \circ \gamma \in C^\infty(\mathbb{R}, \mathcal{M}) \), so \( \Phi \) is smooth. Moreover, \[ \Phi(h_1) = \Phi(h_2) \Rightarrow h_1 = h_2, \] so \( \Phi \) is injective. Next, for \( u \in E \), the tangent map is given by
	\begin{align*}
		(\mathrm{T}_h \Phi)(u)(x) 
		&= \left.\frac{d}{dt}\right|_{t=0} \Phi(h + t u)(x) \\
		&= \left.\frac{d}{dt}\right|_{t=0} \big(h(x) + t u(x), (h(x) + t u(x))^2\big) \\
		&= \big(u(x),\, 2 h(x) u(x)\big).
	\end{align*}
	If \( (\mathrm{T}_h \Phi)(u) = 0 \), then \( u(x) = 0 \) for all \( x \), so \( u = 0 \). Thus, \( \mathrm{T}_h \Phi \) is injective, and \( \Phi \) is an injective immersion.

	It remains to prove that \( \Phi \) is a topological embedding onto its image \( S \), i.e., that \( \Phi\colon E \to S \) is a homeomorphism when \( S \) is endowed with the subspace topology from \( \mathcal{M} \). Consider the following diagram$\colon$
	
	\[
	\begin{tikzcd}[column sep=large, row sep=large]
		E \arrow[r, "\Phi"] \arrow[d, "\mathrm{id}_E"'] & S \arrow[d, hook, "\imath"] \arrow[ld, "\Phi^{-1}"] \\
		E & \mathcal{M} \arrow[l, "\pi_1"]
	\end{tikzcd}
	\]

	Here,
	\( \Phi \colon E \to \mathcal{S} \) is smooth and injective,
	\( \imath \colon S \hookrightarrow \mathcal{M} \) is the inclusion,
	\( \pi_1 \colon \mathcal{M} \to E \) is the projection onto the first component, \( \Phi^{-1}\coloneqq \pi_1 \circ \imath \colon S \to E \) is the inverse map,  constructed by restricting the projection map $\pi_1$ to the subset $S$.
	Next, we prove that the  composition \( \Phi^{-1} \colon S \to E \) has closed graph in \( S \times E \). Hence, by the closed graph theorem, \( \Phi^{-1} \) is continuous. Thus \( \Phi \colon E \to S \) is a homeomorphism.
	
	Let \(((f_n, f_n^2), f_n)\) be a sequence in the graph that converges in \(\mathcal{M} \times E\) to some \(((g, h), f)\).
	We must show that \((g, h) = (f, f^2)\), so that the limit point lies in the graph.
	But since \(f_n \to f\) in \(E\), and the squaring map \(E \to F\), \(f \mapsto f^2\), is continuous \textup{(}being smooth\textup{)}, we have
	\[
	f_n^2 \to f^2 \quad \text{in } F.
	\]
	Hence \((f_n, f_n^2) \to (f, f^2) = (g, h)\), so it must be that \(g = f\), \(h = f^2\).
	Therefore, the limit point is \(((f, f^2), f)\), which lies in the graph.

 Let \( f(x) = (h(x), h(x)^2) \in S \), and consider a smooth curve \( \gamma(t)(x) = (h(x, t), h(x, t)^2) \in S \) with \( h(x, 0) = h(x) \). Then
	\[
	v(x) = \gamma'(0)(x) = (\partial_t h(x, 0), 2 h(x) \partial_t h(x, 0)) = (u(x), 2 h(x) u(x)).
	\]
	Thus,
	\begin{align*}
		\gamma''(0)(x) 
		&= \left( \partial_{tt} h(x, 0), \; 2 (\partial_t h(x, 0))^2 + 2 h(x) \partial_{tt} h(x, 0) \right) \\
		&= \left( \partial_{tt} h(x, 0),\; 2 u(x)^2 + 2 h(x) \partial_{tt} h(x, 0) \right).
	\end{align*}
	The flat spray assigns acceleration \( (0, 0) \), so we must have \( \gamma''(0) = (0, 0) \). Hence, \( \partial_{tt} h(x, 0) = 0 \) and \( u(x)^2 = 0 \), so \( u = 0 \). Therefore, the only vector \( v \) for which \( (f, v, v, (0, 0)) \in \mathrm{T}^2 S \) is \( v = 0 \).
	Thus, 
	\[
	A_{\s, S} = \{ (f, 0) \in \mathrm{T}\mathcal{M} \mid f(x) = (h(x), h(x)^2),\ h \in E \},
	\]
	while the tangent bundle is given by
	\[
	\mathrm{T}S = \{ (f, v) \in \mathrm{T}\mathcal{M} \mid f(x) = (h(x), h(x)^2),\ v(x) = (u(x), 2 h(x) u(x)) \text{ for some } u \in E \}.
	\]
	If \( (f, v) \in A_{\s, S} \), then \( v = 0 \), and the geodesic \( \gamma(t) = f + tv = f \) remains in \( S \). Thus, \( S \) is spray-invariant.
	The tangent bundle \( \mathrm{T}S \) contains non-zero vectors \( v(x) = (h'(x), 2 h(x) h'(x)) \) for non-constant \( h \). Since \( A_{\s, S} \) only contains pairs with \( v = 0 \), we have \( \mathrm{T}S \neq A_{\s, S} \). By Theorem \ref{th:sub}, \( S \) is not totally geodesic.
\end{Exmp}

\begin{cor}\label{cor:1}
	Let \( \fs{M} \) be a manifold equipped with a spray. Assume further that for any two distinct points in \( \fs{M} \), there is unique geodesic connecting them. Let \( S \subset \fs{M} \) be a closed \(C^3\)-submanifold.  Suppose that locally, given two  distinct points in \( S \), the unique geodesic segment in \( \fs{M} \) connecting them lies entirely in \( S \). Then \( S \) is a totally geodesic submanifold of \( \fs{M} \).
\end{cor}
\begin{proof} 
	Let \( p \in S \) and \( v \in \mathrm{T}_pS \). By the hypothesis of the corollary, there exists \( \epsilon > 0 \) such that the unique geodesic  \( \gamma_v\colon (-\epsilon, \epsilon) \to \fs{M} \) with \( \gamma_v(0) = p \) and \( \gamma_v'(0) = v \) is defined. For \( t_0 \in (0, \epsilon) \), let \( q = \gamma_v(t_0) \). Moreover, by the corollary’s hypothesis, the unique geodesic segment \( \gamma_v|_{[0, t_0]} \) connecting \( p \) and \( q \) lies entirely in \( S \).
	
	 By Theorem~\ref{th:1}, since \( \gamma_v(t) \in S \) for all \( t \in [0, t_0] \), we have \( \gamma_v'(t) \in A_{\s, S} \) for all \( t \in [0, t_0] \). In particular, at \( t = 0 \), we have \( v = \gamma_v'(0) \in A_{\s, S} \). Hence, \( \mathrm{T}_pS \subseteq A_{\s, S} \).
	
	Conversely, suppose \( v \in A_{\s, S} \). Let \( \tau(v) = p \in S \). Consider the geodesic \( \gamma_v(t) \) starting at \( p \) with initial tangent \( v \). Since \( v \in A_{\s, S} \), by Theorem~\ref{th:1}, for all \( t \) in the domain of the geodesic where it is defined, we have \( \gamma_v'(t) \in A_{\s, S} \).
	
	Now, let \( q \) be another point in \( S \) such that there is a geodesic  \( \gamma_v \) connecting \( p \) to \( q \), with \( \gamma_v(0) = p \) and \( \gamma_v'(0) = v \). By the local property given in the corollary, this geodesic  lies entirely within \( S \). Since \( \gamma_v(t) \) stays in \( S \), its tangent vector \( \gamma_v'(t) \) must lie in \( \mathrm{T}_{\gamma_v(t)} S \) for all \( t \) 	in its domain.  In particular, at \( t = 0 \), we have \( v = \gamma'(0) \in A_{\s, S} \). Also, since \( \gamma'(0) = v \) and \( \gamma(0) = p \in S \), the initial velocity \( v \) is tangent to \( S \) at \( p \), so \( v \in \mathrm{T}_pS \). This shows that \( A_{\s, S} \subseteq \mathrm{T}S \).
	Therefore, \( A_{\s, S} = \mathrm{T}S \). Thus, by Theorem \ref{th:sub}, \( S \) is a totally geodesic submanifold.
\end{proof}
This result was proven for Banach manifolds using a different technique in  \cite[XI, \S 4, Proposition 4.2]{lang}.
\begin{Exmp}\label{ex:crit}
	Let \( \mathcal{M} = C^\infty(\mathbb{R}^n, \mathbb{R}) \) be the \fr space of smooth real-valued functions on \( \mathbb{R}^n \). The tangent bundle is \( \mathrm{T}\mathcal{M} \cong \mathcal{M} \times \mathcal{M} \). Consider the flat spray
	\(
	\s(f, v) = (f, v, v, \zero{\mathcal{M}}),
	\)
	where \( f, v \in \mathcal{M} \) and \( \zero{\mathcal{M}} \) denotes the zero function. The geodesics are given by
	\(
	\gamma(t) = f + t v.
	\)
	
	Define the subset \( S \subset \mathcal{M} \) as the set of functions that are constant on \( \mathbb{R}^n \), i.e., 
	\[
	S \coloneqq  \left\{ f \in C^\infty(\mathbb{R}^n, \mathbb{R}) \,\middle|\, \exists\, c \in \mathbb{R} \text{ such that } f(x) = c, \ \forall x \in \mathbb{R}^n \right\}.
	\]
	For any two distinct functions \( f_1, f_2 \in \mathcal{M} \), the unique geodesic passing through them is
	\[
	\gamma(t) = f_1 + t(f_2 - f_1).
	\]
	Let \( f_1, f_2 \in S \) be two constant functions, say \( f_1(x) = c_1 \) and \( f_2(x) = c_2 \) with \( c_1 \neq c_2 \). Then for any \( t \in [0, 1] \), the geodesic satisfies
	\[
	\gamma(t)(x) = c_1 + t(c_2 - c_1) = (1 - t)c_1 + t c_2.
	\]
	Notice that for a fixed \( t \), the expression \( (1 - t)c_1 + tc_2 \) yields a single real number that does not depend on \( x \). This means that the function \( \gamma(t) \) takes the same constant value at every point \( x \in \mathbb{R}^n \). Therefore, by the definition of \( S \) as the set of constant functions, \( \gamma(t) \in S \) for all \( t \in [0, 1] \). Thus, the geodesic segment connecting any two points in \( S \) lies entirely in \( S \).
	
	The set \( S \) can be identified with \( \mathbb{R} \) via the constant value. It is a closed linear subspace of \( \mathcal{M} \), and thus a closed \( C^\infty \)-submanifold of \( \mathcal{M} \).
	Since all the conditions of Corollary \ref{cor:1} are satisfied, \( S \)  is a totally geodesic submanifold.
\end{Exmp}

\begin{Rem}
	The local existence of a unique geodesic  in the third condition of Corollary \ref{cor:1}  is crucial for more general manifolds where geodesics might not be straight lines globally.
	In our specific example of constant functions, this local condition happens to hold globally because the geodesics in $\fs{M}$ are straight lines, and any straight line connecting two constant functions consists entirely of constant functions. However, for a general \fr manifold and a submanifold, this containment might only hold for points that are sufficiently close to each other within $S$.
\end{Rem}
\subsection{Automorphisms Preserving Spray Invariance}\label{subsec:1}
In this subsection we study a class of automorphisms of \(\fs{M}\) that preserve spray-invariance.
\begin{lem}
	Let \( \s \) be a spray, and let \( \phi\) be a \(C^{k}\)-automorphism of \(\fs{M}\). Then,  
	\( \phi_{**} \circ \s \circ \phi_{*}^{-1} \) is also a spray.
\end{lem}

\begin{proof}
	Lemma \ref{lem:4} implies that	\( \tilde{\s} = \phi_{**} \circ \s \circ \phi_{*}^{-1} \) is a \(C^{k-2}\)-symmetric second-order vector filed.
	We now need to show that 
	\( \tilde{\s} \) satisfies the spray condition, i.e.,
	\(
	\tilde{\s}(sv) = (L_{\TM{M}})_*(s \tilde{\s}(v)),
	\)
	for all \( s \in \mathbb{R} \) and \( v \in \TM{M} \). Here, $(L_{\TM{M}})_*$ denotes the pushforward of the scalar multiplication map on the tangent bundle $\TM{M}$.
	
	By definition of \( \tilde{S} \), we have
	\(
	\tilde{\s}(sv) = \phi_{**} \circ \s \circ \phi_{*}^{-1}(sv)
	\).	Since \( \phi_{*}^{-1} \) is linear on each fiber (as it is the inverse of the tangent map \( \phi_* \)), we have
	\(
	\phi_{*}^{-1}(sv) = s \phi_{*}^{-1}(v).
	\)
	Substituting this into the expression for \(\tilde{\s}(sv)\), we get
	\[
	\tilde{\s}(sv) = \phi_{**} \circ \s(s \phi_{*}^{-1}(v)).
	\]
	Since \( \s \) is a spray, it satisfies
	\(
	\s(s \phi_{*}^{-1}(v)) = (L_{\TM{M}})_*(s \s(\phi_{*}^{-1}(v))).
	\)
	Substituting this into the expression for \( \tilde{\s}(sv) \), we obtain
	\[
	\tilde{\s}(sv) = \phi_{**} \circ (L_{\TM{M}})_*(s \s(\phi_{*}^{-1}(v))).
	\]
	The pushforward \( \phi_{**} \) commutes with scalar multiplication maps. This is due to the fact that
	\( \phi_{**} \) is linear on each fiber of \(\TTM{M}\). Thus,
	\(
	\phi_{**} \circ (L_{\TM{M}})_* = (L_{\TM{M}})_* \circ \phi_{**}.
	\)
	Applying this, we have
	\[
	\tilde{\s}(sv) = (L_{\TM{M}})_* \circ \phi_{**}(s \s(\phi_{*}^{-1}(v))).
	\]
	Since \( \phi_{**} \) is linear on each fiber,  we can pull out the scalar \(s\), i.e.,
	\[
	\phi_{**}(s \s(\phi_{*}^{-1}(v))) = s \phi_{**} \circ \s \circ \phi_{*}^{-1}(v) = s \tilde{\s}(v).
	\]
	Therefore,
	\(
	\tilde{\s}(sv) = (L_{\TM{M}})_*(s \tilde{\s}(v))
	\).	Thus, \( \tilde{S}  \) satisfies the spray condition. 
\end{proof}
A \(C^{k}\)-automorphism \(\phi\) of \(\fs{M}\) is called an {\it automorphism of the spray} \(\s\) if
\( \phi_{**} \circ \s \circ \phi_{*}^{-1} =\s\). The automorphisms of \(\s\) form a
group under composition called the automorphism group of \(\s\) and denoted
by \(\Aut(\fs{M},\s\)). For finite-dimensional manifolds this concept was introduced in \cite{spf}.
\begin{theorem}\label{th:aut}
	Let \(S \subset \fs{M}\) be a non-empty closed subset that is spray-invariant with respect to \(\s\), and let \(\phi \in \Aut(\fs{M},\s)\). Then \(\phi(S)\) is spray-invariant with respect to \(\s\).
\end{theorem}

\begin{proof}
	Let \(\tilde{p} \in \phi(S)\). Then \(\tilde{p} = \phi(q)\) for some \(q \in S\). Let \(\tilde{v} \in A_{\s, \phi(S)}\) such that \(\tau(\tilde{v}) = \tilde{p}\). Let \(v = \phi_*^{-1}(\tilde{v}) \in \mathrm{T}_q \fs{M}\). Since \(\tau(\tilde{v}) = \phi(q)\), we have 
	\[\tau(v) = \phi^{-1}(\tau(\tilde{v})) = \phi^{-1}(\phi(q)) = q \in S.\]
	We know that \(\tilde{v} \in A_{\s, \phi(S)}\) implies \(\s(\tilde{v}) \in \mathrm{T}^2_{\tilde{p}}\phi(S)\). Using the automorphism property \(\s \circ \phi_* = \phi_{**} \circ \s\), 
	we obtain
	\(
	\s(\tilde{v}) = \s(\phi_{*}(v)) = \phi_{**}(\s(v)).
	\)
	Now, since \( \phi \) maps \( S \) into \( \phi(S) \), its tangent maps satisfy
	\[
	\phi_{*}\colon \TM{S} \to \TM{\phi(S)} \quad \text{and} \quad \phi_{**}\colon \TTM{S} \to \TTM{\phi(S)}.
	\] If \(\s(\tilde{v}) = \phi_{**}(\s(v))\) is tangent to \(\mathrm{T}^2 \phi(S)\) at \(\tilde{p}\), then \(\s(v)\) must be tangent to \(\mathrm{T}^2 S\) at \(q\). Thus, \(v \in A_{\s, S}\).	
	Since \(S\) is spray-invariant and \(v \in A_{\s, S}\), the geodesic \(g\) with \(g(0) = q\) and \(g'(0) = v\) stays in \(S\), i.e., \(g(t) \in S\) for all \(t\) in its domain.
	Now consider the geodesic \(\tilde{g}(t) = \phi(g(t))\). Then
	\[
	\tilde{g}(0) = \phi(g(0)) = \phi(q) = \tilde{p}, \quad \tilde{g}'(0) = \phi_{*}(g'(0)) = \phi_{*}(v) = \tilde{v}.
	\]
	Since \( g(t) \in S \), it follows that \( \tilde{g}(t) = \phi(g(t)) \in \phi(S) \) for all \( t \).
	Hence, the geodesic \( \tilde{g} \) remains in \( \phi(S) \), and therefore \( \phi(S) \) is spray-invariant with respect to \( \s \).
\end{proof}
The \emph{orbit} of a subset \( S \subset \fs{M} \) under the action of \( \Aut(\fs{M}, \s) \) is the set
\[
\mathcal{O}(S) = \{ \phi(S) \mid \phi \in \Aut(\fs{M}, \s) \}.
\]
By Theorem~\ref{th:aut}, each \( \phi(S) \in \mathcal{O}(S) \) is spray-invariant, since automorphisms of \( \s \) preserve the spray structure. Hence, the entire orbit \( \mathcal{O}(S) \) consists of spray-invariant subsets.
\begin{Exmp}\label{ex:9}
	In Example \ref{ex:1}, we showed that for the \fr space 
	\( \fs{E} = C^\infty(\mathbb{R}, \mathbb{R}) \), equipped with the flat spray,
	the set \( \mathsf{S} = \mathsf{S}_{+} \cup \mathsf{S}_{-} \), where
	\[
	\mathsf{S}_{+} \coloneqq  \{ f \in \fs{E} \mid \supp(f) \subseteq [0, \infty) \}, \qquad 
	\mathsf{S}_{-} \coloneqq  \{ f \in \fs{E} \mid \supp(f) \subseteq (-\infty, 0] \}.
	\]
	is a singular spray-invariant. 
	For a fixed \( a \in \mathbb{R} \), \( a \ne 0 \), define the translation map 
	\[
	\phi_a \colon \fs{E} \to \fs{E}, \quad \phi_a(f)(x) = f(x - a),
	\]  
	The induced tangent map \( (\phi_a)_* \) acts on tangent vectors \( v \in \mathrm{T}_f\fs{E} \) as
	\(
	(\phi_a)_*(v)(x) = v(x - a),
	\)
	and similarly for the second tangent map \( (\phi_a)_{**} \). We need to verify
	\(
	(\phi_a)_{**} \circ \s = \s \circ (\phi_a)_*.
	\)
	Indeed,
	\begin{align*}
		(\phi_a)_{**}(\s(f, v)) 
		&= (\phi_a)_{**}(f, v, v, 0) \\
		&= (\phi_a(f), (\phi_a)_*(v), (\phi_a)_*(v), (\phi_a)_*(0)) \\
		&= (f(x - a), v(x - a), v(x - a), 0) \\
		&= \s(f(x - a), v(x - a)) \\
		&= \s(\phi_a(f), (\phi_a)_*(v)) \\
		&= \s((\phi_a)_*(f, v)). 
	\end{align*}
	Thus, \( \phi_a \in \Aut(\fs{E}, \s) \).	
	Since \( \mathsf{S} \) is spray-invariant, by Theorem~\ref{th:aut}, the set
	\[
	\phi_a(\mathsf{S}) = \left\{ g \in \fs{E} \mid \supp(g) \subseteq [a, \infty) \right\} \cup \left\{ g \in \fs{E} \mid \supp(g) \subseteq (-\infty, a] \right\}
	\]
	is a spray-invariant set.
	
	%	\[
	%	\mathcal{O}(\mathsf{S}) = \{ \{ g \in \fs{E}, \mathbb{R}) \mid \supp(g) \subseteq [a, \infty) \} \cup
	% \{ g \in \fs{E} \mid \supp(g) \subseteq (-\infty, a] \} \mid a \in \mathbb{R} \}.
	%	\]
	%	is also spray-invariant.	
\end{Exmp}
\subsection{Orbit Types and Spray Invariance}\label{sub:orbit}
This subsection examines how the symmetries of a manifold, defined by a Lie group action, relate to the invariance of its orbit type decomposition under a \( G \)-invariant spray.

Let \( G \) be a smooth Lie group acting smoothly on a smooth \fr manifold \( \mathsf{M} \) (denoted \( \phi_g \colon \mathsf{M} \to \mathsf{M} \)), for each \( g \in G \), let  the map \( \mathrm{T}_g \) be the {tangent lift} of \( \phi_g \).
A spray \( \s \) on \( \mathsf{M} \) is said to be \emph{\( G \)-invariant} if, for every \( g \in G \), the action of \( g \) on \( \mathsf{M} \) lifts to a smooth transformation \( \mathrm{T}_g \colon \mathrm{T} \mathsf{M} \to \mathrm{T} \mathsf{M} \) such that \( \s \) is preserved under this lifted action. More precisely, for all \( g \in G \), the following diagram commutes:
\[
\begin{tikzcd}
	\mathrm{T} (\mathrm{T} \mathsf{M}) \arrow[r, "\mathrm{T}(\mathrm{T}_g)"] & \mathrm{T} (\mathrm{T} \mathsf{M}) \\
	\mathrm{T} \mathsf{M} \arrow[u, "\s"] \arrow[r, "\mathrm{T}_g"'] & \mathrm{T} \mathsf{M} \arrow[u, "\s"']
\end{tikzcd}
\]
This condition means that for any \( v \in \mathrm{T} \mathsf{M} \), we have
\(
\mathrm{T}(\mathrm{T}_g)(\s(v)) = \s(\mathrm{T}_g(v))
\).

For a point \( x \in \mathsf{M} \), the isotropy group (or stabilizer) of \( x \), denoted by \( G_x \), is the subgroup of \( G \) consisting of all elements \( g \in G \) that leave \( x \) unchanged under the group action, i.e.,
\[
G_x = \{ g \in G \mid g \cdot x = x \}.
\]
A \emph{slice} at \( x \in  \mathsf{M}  \) is a submanifold \( V \subset \mathsf{F} \) containing \( x \) such that
\begin{enumerate}
	\item \( H \)-invariance: \( h \cdot v \in V \) for all \( h \in H \) and \( v \in V \), where \( H = G_x \).
	\item Local triviality: Let $G \times_H V$ be the set of equivalence classes $[g, v]$ obtained from the quotient of the product space $G \times V$ by the right action of $H$, which is defined by $(g, v) \cdot h = (gh, h^{-1} \cdot v)$ for $h \in H$.	There exists a \( G \)-equivariant diffeomorphism
	\[
	\Phi\colon G \times_H V \to U
	\]
	onto a \( G \)-invariant open neighborhood \( U \subset  \mathsf{M}  \) of the orbit \( G \cdot x \), such that \( \Phi([g, v]) = g \cdot v \) and \( \Phi([e, x]) = x \), where \( e \) is the identity in \( G \).
	\item Transversality:
	\begin{enumerate}
		\item \( \mathrm{T}_x V \cap \mathrm{T}_x (G \cdot x) = \{0\} \).
		\item \( \mathrm{T}_x V \) is a closed subspace of \( \mathrm{T}_x  \mathsf{M} \) such that
		\(
		\mathrm{T}_x  \mathsf{M}  = \mathrm{T}_x (G \cdot x) \oplus \mathrm{T}_x V
		\).
		\item The map \( \alpha\colon G \times V \to  \mathsf{M}  \), given by \( \alpha(g, v) = g \cdot v \), has a derivative at \( (e, x) \),
		\[
		\mathrm{T}_{(e, x)} \alpha\colon \mathrm{T}_e G \times \mathrm{T}_x V \to \mathrm{T}_x  \mathsf{M},
		\]
		which is surjective, with kernel complemented in \( \mathrm{T}_e G \times \mathrm{T}_x V \).
	\end{enumerate}
\end{enumerate}

\begin{theorem}\label{th:decom}
	Let \( G \) be a finite-dimensional smooth Lie group acting smoothly on a smooth  \fr  manifold \(  \mathsf{M}  \).
	Assume that a smooth spray \( \s \) on \(  \mathsf{M}  \) is \( G \)-invariant, and that
	for every \( x \in  \mathsf{M}  \), there exists a \( G \)-equivariant neighborhood \( U \) of \( x \) and a \( G \)-equivariant diffeomorphism
	\(
	\Phi\colon G \times_H V \to U
	\)
	where \( V \) is a slice at \( x \) and \( H = G_x \) is the isotropy subgroup.
	Then the orbit type decomposition of \(  \mathsf{M}  \), given by
	\[
	 \mathsf{M}  = \bigcup_{[H]}  \mathsf{M}_{(H)}, \quad \text{where }  \mathsf{M}_{(H)} = \{ x \in  \mathsf{M}  \colon G_x \cong H \},
	\]
	defines a stratification of \(  \mathsf{M}  \) such that each stratum \(  \mathsf{M}_{(H)} \) is spray-invariant.
\end{theorem}

\begin{proof}
	Let \( x \in  \mathsf{M}_{(H)} \). By assumption, there exists a slice \( V \subset  \mathsf{M}  \) at \( x \), and a \( G \)-equivariant diffeomorphism
	\(
	\Phi\colon G \times_H V \to U
	\)
	onto a \( G \)-equivariant open neighborhood \( U \subset  \mathsf{M}  \) of \( G \cdot x \), with \( \Phi([e, 0]) = x \). 
	Define \( \phi = \Phi^{-1} \colon U \to G \times_H V \), and consider the spray
	\(
	\s' \coloneqq (\mathrm{T}\mathrm{T}\phi) \circ \s \circ (\mathrm{T}\phi)^{-1},
	\)
	which is a spray on \( \mathrm{T}(G \times_H V) \). Since both \( \phi \) and \( \s \) are \( G \)-equivariant, the spray \( \s' \) is also \( G \)-invariant.
	Let
	\(
	V_{(H)} := \{ v \in V \colon G_v = H \}
	\)
	denote the set of points in \( V \) with isotropy type \( H \). Then under the diffeomorphism \( \Phi \), we have
	\[
	 \mathsf{M}_{(H)} \cap U = \Phi(G \times_H V_{(H)}).
	\]
	If \( \gamma(t) \) is a geodesic of \( \s \) with \( \gamma(0) = x \in \ \mathsf{M}_{(H)} \) and \( \gamma'(0) \in \mathrm{T}_x \mathsf{M}_{(H)} \), then for small \( t \), we may assume \( \gamma(t) \in U \), so 
	\[
	\phi(\gamma(t)) = [g(t), v(t)] \in G \times_H V.
	\]
	By \( G \)-invariance of \( \s' \), the geodesic \( \gamma(t) \) corresponds to a geodesic \( v(t) \) in \( V \), starting at \( v(0) = 0 \in V_{(H)} \), with tangent vector \( v'(0) \in \mathrm{T}_0 V_{(H)} \). This uses the transversality of the slice, which ensures the splitting
	\[
	\mathrm{T}_x  \mathsf{M}  = \mathrm{T}_x(G \cdot x) \oplus \mathrm{T}_x V,
	\]
	and that \( \gamma'(0) \in \mathrm{T}_x  \mathsf{M}_{(H)} \) implies \( v'(0) \in \mathrm{T}_0 V_{(H)} \).
	
	Now, the induced spray on \( V \) (via projection of \( \s' \)) is \( H \)-invariant (by \( G \)-invariance of \( \s \) and \( H \)-invariance of \( V \)), and since \( v(0) \in V_{(H)} \) and \( v'(0) \in \mathrm{T}_0 V_{(H)} \), the geodesic \( v(t) \) remains in \( V_{(H)} \) for small \( t \). Hence, \( \gamma(t) \in  \mathsf{M}_{(H)} \) for small \( t \), and the set
	\[
	T \coloneqq \{ t \in \operatorname{dom}(\gamma) \colon \gamma(t) \in  \mathsf{M}_{(H)} \}
	\]
	is open and contains \( 0 \). To show that \( T \) is closed, let \( t_n \in T \) be a sequence converging to some \( t_\infty \in \operatorname{dom}(\gamma) \). Since \( \gamma(t_n) \in  \mathsf{M}_{(H)} \) for all \( n \), we know that the isotropy groups satisfy \( G_{\gamma(t_n)} \sim H \). The {isotropy type map}
	\(
	x \mapsto (G_x)
	\) which sends each point in \(  \mathsf{M}  \) to the conjugacy class of its isotropy group, is {upper semi-continuous}. 
	Upper semi-continuity means that if \( x_n \to x \), then (up to conjugacy) we have
	\(
	G_x \supseteq g H g^{-1}\) {for some } \( g \in G,
	\)
	i.e., the isotropy group of the limit can only be {larger} (in the sense of subgroup inclusion), not smaller.
	Applying this to the sequence \( \gamma(t_n) \to \gamma(t_\infty) \), we get
	\[
	G_{\gamma(t_\infty)} \supseteq g H g^{-1} \quad \text{for some } g \in G.
	\]
	
	On the other hand, since the spray \( \s \) is \( G \)-invariant and \( \gamma \) is a geodesic starting at \( \gamma(0) \in \mathsf{M}_{(H)} \), the isotropy type along \( \gamma(t) \) remains constant (i.e., conjugate to \( H \)) by smoothness of the action. Hence,
	\(
	G_{\gamma(t)} \sim H \) for all
	\(t\),
	and in particular,
	\(
	G_{\gamma(t_\infty)} \sim H,
	\)
	so \( \gamma(t_\infty) \in \mathsf{M}_{(H)} \). Thus, \( T \) is closed.
	Since \( T \subseteq \operatorname{dom}(\gamma) \) is both open and closed, and \( \operatorname{dom}(\gamma) \) is connected (being an interval), we conclude that
	\[
	\gamma(t) \in \mathsf{M}_{(H)} \quad \text{for all } t \in \operatorname{dom}(\gamma).
	\]

	Finally, we prove that the orbit type decomposition
	\(
	\mathsf{M} = \bigcup_{[H]} \mathsf{M}_{(H)}
	\)
	is a stratification.  If a point \( y \) is in the closure of a stratum \( \mathsf{M}_{(H)} \), then its isotropy group \( G_y \) must contain a subgroup conjugate to \( H \). In terms of orbit types, this is expressed as \( [G_y] \geq [H] \), where \( [K_1] \geq [K_2] \) if \( K_1 \) contains a subgroup conjugate to \( K_2 \).
	This property arises directly from the {upper semi-continuity of the isotropy group map}. This means that if a sequence \( x_n \in \mathsf{M}_{(H)} \) converges to a point \( y \), then for each \( x_n \), \( G_{x_n} \cong H \). Due to upper semi-continuity, the isotropy group of the limit point \( G_y \) must contain a subgroup conjugate to the isotropy groups of the sequence points. More formally, for any \( y \in \overline{\mathsf{M}_{(H)}} \), there exists some \( g \in G \) such that \( g H g^{-1} \subseteq G_y \).
	Consequently, \( G_y \) contains a subgroup conjugate to \( H \). This implies that the orbit type of \( y \), \( [G_y] \), is larger than or equal to \( [H] \) in the standard partial order of isotropy types. Therefore, \( y \in \mathsf{M}_{(K)} \) for some \( [K] \) such that \( [K] \geq [H] \). Since this holds for any point in the closure of \( \mathsf{M}_{(H)} \), the frontier condition is satisfied.
\end{proof}

\begin{Rem}
	It is important to distinguish between preservation of individual orbits and preservation of orbit type strata under a $G$-invariant spray. Theorem \ref{th:decom} guarantees that geodesics starting in an orbit type stratum remain in that stratum. However, this does not imply that geodesics remain in the same individual orbit.  Thus, spray-invariance applies at the level of strata, not necessarily at the finer level of individual orbits.
\end{Rem}

\section{Spray-Invariant Sets for  \(\mc{k}\)-\fr Manifolds}\label{sec:mc} 
In this section, we work within the category of   \(\mc{k}\)-\fr Manifolds.
We briefly recall the necessary definitions and refer the reader  to \cite{k3, k2, k4a, k4b, kr,k1} for further details.

To define  $\mc{k}$-differentiability (or bounded differentiability), we first introduce the topology of Fr\'{e}chet spaces $\fs{ F}$ and $\fs{E}$ using translation
invariant metric $\mathbbm{m}_{\fs{F}}$ and $\mathbbm{m}_{\fs{E}}$, respectively.  We consider only metrics of the following form:
\begin{equation}
	\mathbbm{m}_{\fs{F}}(x,y) = \sup_{n \in \nn} \dfrac{1}{2^n} \dfrac{\snorm[F,n] {x-y}}{1+\snorm[F,n] {x-y}}.
\end{equation}
Let $\mathsf{L}(\fs{E},\fs{F})$ be the set of all 
linear mappings $ L\colon \fs{E} \rightarrow \fs{F} $ that are (globally) Lipschitz continuous as mappings between metric spaces. Specifically, a linear mapping  $L \in \mathsf{L}(\fs{F},\fs{E})$ satisfies
\begin{equation*}
	\mathrm{Lip} (L )\, \coloneqq \displaystyle \sup_{x \in \fs{E}\setminus\{\zero{\fs{E}}\}} 
	\dfrac{\mathbbm{m}_{\fs{F}}( L(x),\zero{\fs{F}})}{\mathbbm{m}_{\fs{E}}( x,\zero{\fs{E}})} < \infty.
\end{equation*}
We define a topology on $\mathsf{L}(\fs{E},\fs{F})$ using the following  translation invariant metric:
\begin{equation}  
	\mathsf{L}(\fs{E},\fs{F}) \times \mathsf{L}(\fs{E},\fs{F}) \longrightarrow [0,\infty) , \,\,
	(L,H) \mapsto \mathrm{Lip}(L-H),
\end{equation}
where \(\mathrm{Lip}(L - H)\) denotes the Lipschitz constant of the linear map \(L - H\).

Let \( \varphi\colon U \opn \fs{E} \to \fs{F} \) be a \(C^1\)-mapping.  
If \( \dd \varphi(x) \in \mathsf{L}(\fs{E}, \fs{F}) \) for all \( x \in U \),  
and the induced map  
\[
\dd \varphi\colon U \to \mathsf{L}(\fs{E}, \fs{F}), \quad x \mapsto \dd \varphi(x)
\]  
is continuous, then \(\varphi\) is called bounded differentiable or \(\mc{1}\).  
Mappings of class \(\mc{k}\), for \(k > 1\), are defined recursively.  
An \(\mc{k}\)-\fr manifold is a \fr manifold whose coordinate transition functions are all \(\mc{k}\)-mappings.

Let $(B_1, \mid \cdot \mid_1)$ and $(B_2, \mid \cdot \mid_2)$ be Banach spaces. A linear operator $T: B_1 \to B_2$
is called nuclear if it can be written in the form
\(
T(x) = \sum_{j=1}^{\infty} \lambda_j \langle x,x_j \rangle y_j
\),
where $\langle \cdot, \cdot \rangle$ is the duality pairing between $B_1$ and its dual $(B_1', \mid \cdot \mid'_1)$, $x_j \in B_1'$ with
$\mid x_j \mid_1' \leq 1$, $y_j \in B_2$ with $\mid y_1 \mid_2 \leq 1$, and $\lambda_j$ are complex numbers such that $\sum_j \mid \lambda_j \mid < \infty$.

For a seminorm \(\snorm[\fs{F},i]{\cdot}\) on \(\fs{F}\), we denote by \(\fs{F}_i\) the Banach space given by completing \(\fs{F}\) using the seminorm  
\(\snorm[\fs{F},i]{\cdot}\). There is a natural map from \(\fs{F}\) to \(\fs{F}_i\) whose kernel is \(\ker \snorm[\fs{F},i]{\cdot}\).  

A Fr\'{e}chet space \(\fs{F}\) is called nuclear if for any seminorm \(\snorm[\fs{F},i]{\cdot}\), we can find a larger seminorm \(\snorm[\fs{F},j]{\cdot}\)  
so that the natural induced map from \(\fs{F}_j\) to \(\fs{F}_i\) is nuclear.  
A nuclear Fr\'{e}chet manifold is a manifold modeled on a nuclear Fr\'{e}chet space.
A key feature of Fr\'{e}chet nuclear spaces is that they have the Heine-Borel property. This provides a significant advantage over Banach spaces, as no infinite-dimensional Banach space is nuclear.  

In Definition \ref{def:fis}, we introduced the concept of a spray-invariant set. This notion has an analogous definition for vector fields on a manifold. The following definition, applicable to both \( \mc{k} \)-Fr\'{e}chet manifolds and  $C^k$-Fr\'{e}chet manifolds, shares the same underlying structure as Definition \ref{def:fis}.

In this section, we assume that \(\fs{M}\) is an \(\mc{k}\)-\fr manifold with \(k\geq 4\), modeled on \(\fs{F}\). 

\begin{Defn} [Definition 3.1, \cite{kr}]\label{def:fis2}
	Let \( A \subset \fs{M} \) and \( \vf \) be an \( \mc{1} \)-vector field on \( \fs{M} \). The set \( A \) is called flow-invariant with respect to \( \vf \) if, for any integral curve \( {I}(t) \) of \( \vf \) with \( {I}(0) \in A \), we have \( {I}(t) \in A \) for all \( t \geq 0 \) within the domain of \( {I} \).
\end{Defn}

\begin{theorem} [Theorem 3.2, Nagumo-Brezis Theorem, \cite{kr}]     \label{th:nb}
	Let \( \fs{M} \) be a nuclear \( \mc{k} \)-Fr\'{e}chet manifold, and let \( \vf\colon \fs{M} \to \TM{M} \) be an \( \mc{1} \)-vector field. Let \( A \subset \fs{M} \) be closed. Then, \( A \) is flow-invariant with respect to 
	\( \vf \) if and only if for each \( x \in \fs{M} \), there exists a chart \( (U, \phi) \) around \( x \), such that
	\begin{equation}
		\lim_{ t \to 0^+} t^{-1} \mathbbm{m}_{\fs{F}} \left( \phi(x) + t \dd\phi(x)(\vf(x)), \phi(U \cap A) \right) = 0.
	\end{equation}
\end{theorem}

Lemma \eqref{lem:2}, which establishes the chart-independence of first-order adjacent tangency, ensures that the condition in Theorem \ref{th:nb} is independent of the choice of chart. This result, not proved in \cite{kr}, provides additional strength to the theorem.
\begin{theorem}\label{th:imp}
	Let \( \fs{M} \) be a nuclear \( \mc{k} \)-Fr\'{e}chet manifold, and let \( S \subset \fs{M} \) be a subset such that \( A_{\s, S} \) is non-empty and closed. Then, the following are equivalent\(\colon\)
	\begin{enumerate}
		\item \( S \) is spray-invariant with respect to \( \s \).
		\item \( \s \) is adjacent tangent to \( A_{\s, S} \) when regarded as a vector field on \( \TM{M} \).
	\end{enumerate}
\end{theorem}

\begin{proof}
	\textbf{(1) \(\Rightarrow\) (2):} By Theorem \ref{th:1}, spray-invariance of \( S \) implies that all  geodesics whose initial tangent vectors are  in \( A_{\s, S} \) remain within it. The Nagumo-Brezis condition (Theorem \ref{th:nb}) then guarantees the adjacent tangency
	\[
	\lim_{t \to 0^+} t^{-1} \mathbbm{m}_{\fs{F}}\left(\phi(v) + t \dd\phi(v)(\s(v)), \phi(U \cap A_{\s, S})\right) = 0 \quad \forall v \in A_{\s, S}.
	\]
	
	\textbf{(2) \(\Rightarrow\) (1):} If \( \s \) is adjacent tangent to \( A_{\s, S} \), applying Theorem \ref{th:nb} to $\TM{M}$ with $A_{\s, S} $ as the closed subset implies \( A_{\s, S} \) is spray-invariant. 
\end{proof}
In the rest of this subsection, we assume that \(\fs{M}\) is second countable, a property essential for applying  transversality.
The notion of transversality extends  to \(\mc{k}\)-\fr manifolds and has been explored in \cite{k3}. Here, we summarize the results relevant to our discussion.

Let $\varphi\colon\fs{M} \to \fs{N}$ be an $\mc{r}$-mapping,  where $ r \geq 1 $, and $S \subseteq \fs{N}$ a submanifold.
We say that $\varphi$ is transversal to \(S\) , denoted by $\varphi \pitchfork S$, if
either $ \varphi^{-1}(S) = \emptyset $, or, if for each $ x \in \varphi^{-1}(S) $, the following conditions hold:
\begin{enumerate}
	\item \( (\mathrm{T}_x \varphi)(\mathrm{T}_x \fs{M}) + \mathrm{T}_{\varphi(x)} S = \mathrm{T}_{\varphi(x)} \fs{N} \), and
	\item \( (\mathrm{T}_x \varphi)^{-1}(\mathrm{T}_{\varphi(x)} S) \) splits in \( \mathrm{T}_x \fs{M} \).
\end{enumerate}

The proof of the following lemma can be readily adapted from the case of Banach manifolds  (see \cite{mon1}) to our setting, so we omit it here.
\begin{lem}\label{lem:ld}
	Let 
	$\va\colon \fs{M} \to \fs{N}$ be an $\mc{k}$ mapping between $\mc{k}$-\fr manifolds $\fs{M}$ and $\fs{N}$, and let $W \subset \fs{N}$ be an $\mc{k}$-submanifold of $\fs{N}$. Then
	\[
	\va \pitchfork W \quad \iff \quad \mathrm{T}\va \pitchfork TW.
	\]
\end{lem}

\begin{theorem}[Theorem 2.2, Transversality Theorem,\cite{k3}]\label{th:t1}
	Let $ \varphi \colon \fs{M} \to \fs{N} $ be an 
	$\mc{r}$-mapping with $r\geq1$, and let $ S \subset N $ be an $\mc{r}$-submanifold such that $ \varphi \pitchfork S $. Then, $ \varphi^{-1}(S) $ is either empty or an $\mc{r}$-submanifold of $ \fs{M} $
	with
	\[
	(T_x \varphi)^{-1}(T_yS) = T_x (\varphi^{-1}(S)), \quad x \in \varphi^{-1}(S),\, y=\varphi(x).
	\]
	If $ S $ has finite co-dimension in $ \fs{N} $, then $\codim (\varphi^{-1}(S)) = \codim S  $. Moreover, 
	if $\dim S = m <\infty$  and $\varphi$ is an $\mc{r}$-\li-Fredholm mapping of index $ l $, 
	then $\dim \varphi^{-1}(S) = l+m$. 
\end{theorem}
Let 
$\va\colon \fs{M} \to \fs{N}$ be an $\mc{3}$-mapping between $\mc{4}$-\fr manifolds $\fs{M}$ and $\fs{N}$, and let $W \subset \fs{N}$ be an $\mc{3}$-submanifold of $\fs{N}$ such that \(\va \pitchfork W\). 
Then, by the transversality theorem, \(S=\va^{-1}(W)\)
is an $\mc{3}$-submanifold of $\fs{M}$, and 
\(
\mathrm{T}S = (\mathrm{T}\va)^{-1}(\mathrm{T}W)
\).
Since Lemma \ref{lem:ld} implies \( \mathrm{T}\va \pitchfork \mathrm{T}W \), applying the transversality theorem again yields
\[
\mathrm{T}(\mathrm{T}S) = (\mathrm{T}(\mathrm{T}\va))^{-1}(\mathrm{T}(\mathrm{T}W)).
\]
Consequently, for a given spray \( \s \) on \( \fs{M} \), Equation \eqref{eq:lm} implies
\[
A_{\s, S} = (\mathrm{T}(\mathrm{T}\va) \circ \s)^{-1}(\mathrm{T}(\mathrm{T}W)).
\]
Suppose $ \fs{F}_1 $ is a closed subset of the \fr space $ \fs{F} $ that splits it. Let $ \fs{F}_2 $
be one of its complements, i.e., $ \fs{F} = \fs{F}_1 \oplus \fs{F}_2 $. Let $ {S} $ be an \(\mc{k}\)-submanifold modeled on $ \fs{F}_1 $.
\begin{theorem}\label{th:tf}
	Let \( \fs{M} \) be a nuclear \( \mc{k} \)-Fréchet manifold, and let \( S \) be the submanifold of \( \fs{M} \) introduced above. 
	If \( S \) is a closed \(\mc{3}\)-submanifold of \( \fs{M} \) such that \( \s \big|_{\mathrm{T}S} \pitchfork \mathrm{T}(\mathrm{T}S) \), 
	then \( S \) is spray-invariant with respect to \( \s \) if and only if 
	\begin{equation}\label{eq:lss}
		\forall v \in \s(\mathrm{T}(\mathrm{T}S)), \quad \dd \s(v)(\s(v)) \in \mathrm{T}_{\s(v)}(\mathrm{T}(\mathrm{T}S)).
	\end{equation}
\end{theorem}
\begin{proof}
	Define \( \bl{T}(\mathrm{T}S) \) as the set of elements \( w \in \TTM{M}) \) such that \( \tau_2(w) \in \mathrm{T}S \), and there exists a chart \( \phi \colon U \to \fs{F} \) at \( \tau(\tau_2(w)) \) satisfying the following conditions:
	\begin{itemize}
		\item  \( \phi(U \cap S) = \phi(U) \cap \fs{F}_1 \),
		\item  \( \dd(\dd\phi)(\tau_2(w)) (w) \in \fs{F}_1 \times \fs{F} \).
	\end{itemize}
	This definition is independent of the choice of chart. The definition directly implies
	\[
	\dd(\dd \phi)\big(\mathrm{T}(\mathrm{T}U)\big) \cap \bl{T}(\mathrm{T}S))\big) = (\phi(U) \cap \fs{F}_1 ) \times \fs{F}_1  \times \fs{F}_1  \times \fs{F}.
	\]  
	This implies that \( \bl{T}(\mathrm{T}S) \) is a submanifold of \(\TM(\mathrm{T}\fs{M})\) modeled on \( \fs{F}_1  \times \fs{F}_1  \times \fs{F}_1  \times \fs{F}\). Moreover,  since \(\s\) on \(\fs{M}\) maps \(\mathrm{T}S \) into \( \bl{T}(\mathrm{T}S) \), and
	\[
	\dd(\dd\phi) \circ (\s \big|_{\mathrm{T}S}) \circ (\dd \va)^{-1}\big(\phi(U) \cap \fs{F}_1 \big) \times \fs{F}_1 \subset \fs{F}_1  \times \fs{F}_1  \times \fs{F}_1  \times \fs{F}.
	\] 
	we find that the image of \(\s \big|_{\mathrm{T}S} \) lies in \( \bl{T}(\mathrm{T}S)\). Now, the transversality assumption implies
	\[
	\dd \Big(\s \big|_{\mathrm{T}S}(v)\Big) (\TS{v} (\mathrm{T}S )) + \mathrm{T}_{\s(v)}  (\mathrm{T}(\mathrm{T}S)) = \mathrm{T}_{\s(v)} (\bl{T}(\mathrm{T}S )), \quad \text{ for } v \in \s^{-1}(\mathrm{T}(\mathrm{T}S )).
	\]
	
	Therefore, by Equation \eqref{eq:lm} and Theorem~\ref{th:imp}, \(A_{\s, S} = \s^{-1}(\mathrm{T}(\mathrm{T}S))\)
	is an \(\mc{1}\)-\fr submanifold of \(\mathrm{T}S \), and  its tangent space at \(v \in A_{\s, S}  \) is given by
	\[
	\TS{v} (A_{\s, S} ) = \dd \s(v)^{-1}\Big(\mathrm{T}_{\s(v)} (\mathrm{T}(\mathrm{T}S ))\Big).
	\]
	Consequently, by Theorem \ref{th:imp}, \(S\) is spray-invariant with respect to \(\s\) if and only if
	\[
	\forall v \in  A_{\s, S} ,\quad \s(v) \in \TS{v} (A_{\s, S} )
	\]
	which is equivalent to the condition stated in \eqref{eq:lss}.\label{sec:bh}
\end{proof}
\begin{Rem}
	In Theorem \ref{th:tf}, explicitly verifying the transversality condition  can be very difficult. The infinite-dimensional nature of \( \mathrm{T}(\mathrm{T}S) \), together with the complexity of identifying suitable complements in the modeling space, poses significant analytical challenges even in relatively simple settings.
\end{Rem}
\section{Aspects of Banach and Hilbert Manifolds}\label{sec:hb}
In contrast to  \fr manifolds, for Banach manifolds there is a well-developed framework for the existence, uniqueness, and regularity of ordinary differential equations. This allows for the application of tools such as geodesic flows to characterize invariance.

We use the same notations as before. Regarding differentiability, Definition  \ref{def:diff} applies to Banach spaces as well; however, Banach spaces admit an equivalent formulation (see \cite{lang}).

In Section \ref{sec:2}, Definitions \ref{def:admiset} and \ref{def:fis}, along with Theorems \ref{th:1}, \ref{th:sub}, \ref{th:decom},  and \ref{th:aut}, and their consequences, remain valid for Banach manifolds as well. This follows from the fact that all prerequisite results hold in the Banach setting. In particular, relevant properties of sprays are discussed in \cite{lang}, while adjacent cones are treated in \cite{mon}.

In Section \ref{sec:mc}, an analogous  of Theorem \ref{th:tf} holds for arbitrary Banach manifolds, since the transversality theorem is available in this context. However, as previously observed, verifying the transversality condition remains challenging even for Banach and Hilbert manifolds. 

Theorem \ref{th:imp} relies on the Nagumo-Brezis Theorem for nuclear manifolds. However, no infinite-dimensional Banach manifold is nuclear. Nevertheless, a variant of the Nagumo-Brezis Theorem is available for arbitrary Banach manifolds of class \( C^k \), with \( k \geq 2 \); see \cite{pa}. Thus, Theorem \ref{th:imp} holds for arbitrary Banach manifolds of class at least \( C^4 \).

\begin{theorem}\label{th:impb}
	Let \( \fs{B} \) be a  \( C^k \)-Banach manifold, \(k\geq 4\), and \( S \subset \fs{B} \) a subset such that \( A_{\s, S} \) is non-empty and closed. Then,  \( S \) is spray-invariant if and only if
	\( \s \) is adjacent tangent to \( A_{\s, S} \) when regarded as a vector field on \( \TM{B} \).
\end{theorem}
\begin{Exmp}
	\label{ex:nonneg}
	Consider the Banach manifold \( \mathcal{M} = C^k(S^1, \mathbb{R}) \) equipped with the flat spray \( \s(f, v) = (f, v, v, 0) \), whose geodesics are affine paths \( \gamma(t) = f + tv \). 
	
	Let $N$ be a fixed non-negative integer. Define the set $S$ as follows
	\[
	S \coloneqq \left\{ f \in \mathcal{M} \mid f(\theta) = a_0 + \sum_{j=1}^{N} (a_j \cos(j\theta) + b_j \sin(j\theta)) \text{ for some } a_j, b_j \in \mathbb{R} \right\}.
	\]
	The set $S$ is a finite-dimensional linear subspace of $\mathcal{M}$ and therefore a smooth submanifold. Thus, the adjacent cone $\mathrm{T}_fS$ at a point $f \in S$ is the tangent space, i.e.,
	\(
	\mathrm{T}_fS = S
	\).
	
	The admissibility condition for a velocity $v$ at a point $f \in S$ is simply $v \in \mathrm{T}_fS = S$. Thus, the admissible set is
	\[
	A_{\s, S} = \{ (f, v) \in \mathrm{T}\mathcal{M} \mid f \in S, v \in S \} = S \times S.
	\]
	As a finite-dimensional subspace, $S$ is a closed subset of $\mathcal{M}$. Consequently, the product set $S \times S$ is a closed subset of the product space $\mathcal{M} \times \mathcal{M}$. Therefore, $A_{\s, S}$ is closed.
	
	The spray $\s$ assigns the vector $(v, 0)$ to the point $(f, v) \in \mathrm{T}\mathcal{M}$. We show that this vector is in the tangent cone $\mathrm{T}_{(f,v)}A_{\s, S}$. Since $A_{\s, S} = S \times S$ is a linear subspace, its tangent cone at any point is the space itself. We check if $(v, 0) \in S \times S$. This requires $v \in S$ and $0 \in S$. Both are true since $(f,v) \in A_{\s, S}$ and $S$ is a linear subspace containing the zero function. Thus, the spray is adjacent tangent to $A_{\s, S}$. Therefore by Theorem~\ref{th:impb} the subspace $S$ is spray-invariant.
\end{Exmp}
We assume that \( \fs{B} \) is a Banach manifold of class \(C^k\) with \(k\geq 4\),  and that \( \s \) is a spray on \( \fs{B}\) of class  \(C^2\). Recall that the geodesic flow is the  mapping  
\(
\Phi_t \colon \mathrm{T}\fs{B} \to \mathrm{T}\fs{B}
\) that satisfies
\(
\Phi_t(v)  = g_v'(t),
\)
where \( g_v\colon I \to \fs{B} \) is the unique geodesic with initial tangent \( v \in \mathrm{T}\fs{B} \).  

\begin{theorem}\label{th:geod}
	A closed subset \( S \subset \fs{B} \) is spray-invariant  if and only if its admissible set \( {A}_{\s, S} \) is invariant under the geodesic flow $\Phi_t$.
\end{theorem}
\begin{proof}
	Assume \( S \) is spray-invariant. Let \( v \in \Adm{\s}{S} \). By definition of the admissible set, the geodesic \( \gamma_v(t) = \tau(\GeoFlow{t}(v)) \) satisfies \( \gamma_v(t) \in S \) for all \( t \) in its maximal interval \( I \). By Theorem \ref{th:1}, the tangent field \( \gamma_v'(t) = \GeoFlow{t}(v) \) remains in \( \Adm{\s}{S} \). Thus, \( \GeoFlow{t}(v) \in \Adm{\s}{S} \) for all \( t \in I \), proving \( \Adm{\s}{S} \) is \( \GeoFlow{t} \)-invariant.
	
	Conversely, assume \( \Adm{\s}{S} \) is \( \GeoFlow{t} \)-invariant. Let \( \gamma \colon I \to \fs{B} \) be a geodesic with \( \gamma(0) \in S \) and \( \gamma'(0) \in \Adm{\s}{S} \). By spray invariance we have
	\[
	\forall t \in I, \quad \gamma'(t) = \GeoFlow{t}(\gamma'(0)) \in \Adm{\s}{S}.
	\]
	Then Theorem \ref{th:1} implies \( \gamma(t) = \tau(\gamma'(t)) \in S \) for all \( t \in I \). Hence, \( S \) is spray-invariant.
\end{proof}

The spray \(\s\) induces a unique torsion-free covariant
derivative \(\nabla^{\fs{B}}\) (VIII, \S 2, Theorem 2.1,\cite{lang}). Let \( g\colon I \to \fs{B} \) be a \( C^2 \)-curve. We say that a lift \( \gamma\colon I \to \mathrm{T}\fs{B} \) of \( g\) is \( g\)-parallel if  \(\nabla^{\fs{B}}_{g'}\gamma =0\). A curve $g$ is a geodesic for the spray if and only if \(\nabla^{\fs{B}}_{g'}g' =0\), that is, if
and only if \(g'\) is \(g\)-parallel.

Manifolds modeled on self-dual Banach spaces, including Hilbert spaces, admit canonical sprays induced by pseudo-Riemannian metrics (VIII, \S 7, Theorem 7.1, \cite{lang}). This theorem also holds for Hilbert Riemannian manifolds, as the proof does not rely on the indefiniteness of the pseudo-Riemannian metric. Instead, it depends only on the metric being smooth and non-degenerate, properties that Riemannian metrics also possess. 

Consider canonical sprays on Hilbert Riemannian manifolds.
Suppose that \(\fs{H}\) is a Hilbert  Riemannian manifold and that \( S \subset \fs{H} \) is a 
\(C^1\)-submanifold with the induced metric (or Levi-Civita) covariant derivative  \( \nabla^{{S}} \) defined by canonical spray \(\s\). There exists a canonical symmetric bilinear bundle map, known as the \emph{second fundamental form} (see~\cite[IX, \S 1, Propositions 1.2 and 1.3]{lang}). This map is given by the Gauss formula as follows
\[
\nabla^{\fs{H}}_X Y_x(x) = \nabla^S_X Y(x) + \mathrm{II}(X(x), Y(x)),
\]
for any \(x \in S\) vector fields \( X, Y \) of \( S \) near \(s\), and the extension \(Y_x\) of \(Y\) near \(x\).

Suppose that \( S \subset \fs{H} \) is spray-invariant, and Let \( \gamma \colon I \to \fs{H} \) be a geodesic with \( \gamma(0) \in S \) and \( \gamma'(0) \in \Adm{\s}{S} \). Then
\[
0=\nabla_{{\gamma'}(t)} {\gamma'}(t)  \text{ in } T_{\gamma(t)}{S} \quad \forall t \in I.
\]
By the Gauss formula
\[
\nabla^{\fs{H}}_{{\gamma}'} {\gamma'} = \nabla^{\mathcal{S}}_{{\gamma}'} {\gamma'} + \mathrm{II}({\gamma'}, {\gamma}'),
\]
since the total derivative is tangent to \( \mathcal{S} \), its normal component must vanish, i.e.,
\(
\mathrm{II}({\gamma'(t)}, {\gamma'(t)}) = 0
\)
for all \(t \in I\).
A polarization identity is given by
\[
\mathrm{II}(X, Y) = \frac{1}{2} \left( \mathrm{II}(X+Y, X+Y) - \mathrm{II}(X, X) - \mathrm{II}(Y, Y) \right).
\]
If this identity could be applied for arbitrary \( X, Y \in \mathrm{T}S \), then the vanishing of \( \mathrm{II}(X,X) \) would imply the vanishing of \( \mathrm{II}(X,Y) \).
However, spray-invariance only gives us the condition \( \mathrm{II}(Z, Z) = 0 \) for vectors \( Z \) in \(A_{\s,S}\). It does not guarantee that \( X + Y \) is also such a tangent vector, and hence we cannot conclude that \( \mathrm{II}(X+Y, X+Y) = 0 \) unless 
\(A_{\s,S}=\mathrm{T}S\).

\begin{Exmp}\label{ex:hi}
	Consider the Hilbert manifold \( \mathcal{M} = H^1(S^1, S^2) \), the space of maps from the circle \( S^1 \) into the 2-sphere \( S^2 \) whose first derivatives are square-integrable. By the Sobolev embedding theorem \textup{(}$1 - \dfrac{1}{2}> 0$\textup{)}, every map in this space is continuous.

	Let \( \s_{S^2} \) be the canonical geodesic spray on the finite-dimensional manifold $S^2$. 
	We define a spray \( \s \) on the loop space \( \mathcal{M} \)  by applying the target spray pointwise. For any $(f, v) \in \mathrm{T}\mathcal{M}$, the spray $\s(f,v)$ is the second-order vector field along $f$ given by
	\[
	\s(f, v)(\theta) \coloneqq \s_{S^2}(f(\theta), v(\theta)).
	\]
	The geodesics of \( \s \) are defined as the paths \( \gamma(t) \) in \( \mathcal{M} \) that satisfy the pointwise geodesic equation for $S^2$:
	\[
	\nabla_{\gamma'(\theta)}^{S^2} \gamma'(\theta) = 0 \quad \text{for each } \theta \in S^1.
	\]
	Since unique solutions for this ODE exist on the compact manifold $S^2$ for any initial condition, this spray is well-defined. Its geodesics are, by construction, precisely the pointwise geodesics of the target manifold $S^2$.
	
	Let \( C \subset S^2 \) be a great circle, which is a totally geodesic submanifold. Define the subset of constant loops on this circle:
	\[
	S \coloneqq \left\{ f \in \mathcal{M} \,\middle|\, \exists p \in C \text{ such that } f(\theta) = p \text{ for all } \theta \in S^1 \right\}.
	\]
	This set $S$ is a closed $C^\infty$-submanifold of $\mathcal{M}$. The tangent space \( \mathrm{T}_f S \) at a point \( f(\theta) = p \in S \) consists of constant vector fields \( v(\theta) = v_0 \) where \( v_0 \in \mathrm{T}_p C \).

	We first determine the admissible set $A_{\s, S}$. A vector $v = (f, u) \in \mathrm{T}S$ is in $A_{\s, S}$ if the acceleration of its geodesic, $\gamma_v''(0)$, is tangent to $S$. For the spray $\s$, this acceleration is computed pointwise: 
	\[
		\gamma_v''(0)(\theta) = \nabla_{u(\theta)}^{S^2} u(\theta).
	\]
	Since $u(\theta) = u_0$ is a constant vector tangent to the great circle $C$, and $C$ is itself a geodesic on $S^2$, the self-covariant derivative $\nabla_{u_0}^{S^2} u_0$ is zero. The zero vector field is tangent to $S$. This condition holds for all vectors $v \in \mathrm{T}S$. Therefore, the admissible set is the entire tangent bundle of $S$:
	\[
	A_{\s, S} = \mathrm{T}S.
	\]
Therefore, by  Theorem~\ref{th:sub} for Hilbert manifolds, $S$ is a totally geodesic submanifold.
\end{Exmp}

\begin{Exmp}\label{ex:stra}
	Let \( H = \ell^2 \), the separable Hilbert space of square-summable sequences with standard inner product
	\[
	\langle x, y \rangle = \sum_{i=1}^\infty x_i y_i,
	\]
	and let \( \{e_n\}_{n \in \mathbb{N}} \) denote its standard orthonormal basis.
	Define the subset
	\[
	\mathcal{S} \coloneqq \left\{ x \in H \;\middle|\; \text{only finitely many coordinates of } x \text{ are nonzero} \right\}.
	\]
	This is the space of finite sequences, and can be expressed as a countable union\(\colon\)
	\[
	\mathcal{S} = \bigcup_{k=1}^{\infty} H_k, \quad \text{where } H_k \coloneqq  \operatorname{span}(e_1, \dots, e_k).
	\]
	Each \( H_k \) is a finite-dimensional linear subspace of \( H \). Consider the flat spray of \(\ell^2 \).
	Let \( x \in \mathcal{S} \) and \( v \in \mathrm{T}_x \mathcal{S} \). Then there exists \( k \) such that both \( x, v \in H_k \). The geodesic starting at \( x \) with tangent \( v \) is given by
	\(
	\gamma(t) = x + tv
	\). Since \( H_k \) is a linear subspace, \( \gamma(t) \in H_k \subset \mathcal{S} \) for all \( t \in \mathbb{R} \). Thus,
	\( \mathcal{S} \) is spray-invariant.
	The set \( \mathcal{S} \) is not a smooth submanifold of \( H \), since it is not locally homeomorphic to a Hilbert space. It is a stratified space, built from the smooth finite-dimensional submanifolds \( H_k \).  
	We consider a {stratification} of $\mathcal{S}$ into strata $S_i$:
	\[
	S_i = H_i \setminus H_{i-1}.
	\]
	Let \( S_i \) and \( S_j \) be  two strata. We consider the following cases$\colon$
	\begin{itemize}
		\item \textbf{Case 1: \( i < j \)} \\
		\( \overline{S_i} = H_i \). Since \( H_i \subset H_j \), but \( H_i \) contains vectors with at most \( i \) nonzero components, while \( S_j \) contains vectors with exactly \( j > i \) nonzero components, it follows that \( H_i \cap S_j = \emptyset \). Thus, \( \overline{S_i} \cap S_j = \emptyset \). 
		\item \textbf{Case 2: \( i = j \)} \\
		Trivially, \( \overline{S_i} = H_i \), and \( \overline{S_i} \cap S_i = S_i \neq \emptyset \). Furthermore, \( S_i \subset \overline{S_i} \) by definition.
		\item \textbf{Case 3: \( i > j \)} \\
		We have \( H_j \subset H_i \), and \( S_j = H_j \setminus H_{j-1} \subset H_i \). Hence, \( \overline{S_i} \cap S_j = S_j \neq \emptyset \), and \( S_j \subset \overline{S_i} \).
	\end{itemize}
	In all cases, the frontier condition is satisfied for the decomposition \( \mathcal{S} = \bigsqcup_{k=1}^\infty S_k \). Thus, this decomposition defines a stratification of \( \mathcal{S} \).
	
	Each \( H_k \) is totally geodesic in \( H \) due to the flatness of the ambient geometry. However, the union \( \mathcal{S} \) is not totally geodesic as a whole, since it lacks a global smooth structure: the second fundamental form is not defined across strata.
	
\end{Exmp}

\end{document}